\documentclass[a4paper,12pt]{article}

\pdfoutput=1

\usepackage{amsmath, amssymb, amsthm, mathrsfs}
\usepackage{color}
\usepackage{hyperref}

\begin{document}
\newtheorem{theorem}{Theorem}[section]
\newtheorem{lemma}[theorem]{Lemma}
\newtheorem{proposition}[theorem]{Proposition}
\theoremstyle{definition}
\newtheorem{remark}[theorem]{Remark}
\newtheorem{definition}[theorem]{Definition}
\numberwithin{equation}{section}


\newcommand{\choice}[2]{#1}    


\allowdisplaybreaks[3]
\baselineskip 16pt

\title{Low regularity a priori estimate for KDNLS via the short-time Fourier restriction method%
\choice%
{\thanks{A shorter version of this article has been published in {\it Adv.~Contin.~Discrete Models} 2023, 10 (https://doi.org/10.1186/s13662-023-03756-6), in which part of the proof of Proposition~\ref{prop:R6} was omitted.
In the present version, we include a complete proof of Proposition~\ref{prop:R6} and also add the following contents in Section~\ref{sec:2}: Definition~\ref{def:UpVp}, Lemma~\ref{lem:Up-restriction}, and a discussion after Lemma~\ref{lem:L6} with a proof of the estimate \eqref{est:disp}.}%
}{%
}%
}

\author{Nobu KISHIMOTO\thanks{Research Institute for Mathematical Sciences, Kyoto University, Kyoto 606-8502, JAPAN} ~and~ Yoshio TSUTSUMI\thanks{Institute for Liberal Arts and Sciences, Kyoto University, Kyoto 606-8501, JAPAN}\\[15pt]}
\date{Dedicated to the memory of Professor Jean Ginibre}

\maketitle

\begin{abstract}
In this article, we consider the kinetic derivative nonlinear Schr\"odinger equation (KDNLS), which is a one-dimensional nonlinear Schr\"odinger equation with a cubic derivative nonlinear term containing the Hilbert transformation.
For the Cauchy problem both on the real line and on the circle, we apply the short-time Fourier restriction method to establish \emph{a priori} estimate for small and smooth solutions in Sobolev spaces $H^s$ with $s>1/4$.
\end{abstract}



\section{Introduction}

In the present article, we continue our study in \cite{KT22,KT-gauge} and consider the kinetic derivative nonlinear Schr\"odinger equation (KDNLS) on $\mathbb{R}$ and on $\mathbb{T}:=\mathbb{R}/2\pi \mathbb{Z}$:
\begin{equation}\label{kdnls}
\partial_t u = i \partial_x^2 u + \alpha \partial_x \big( |u|^2u\big) + \beta \partial_x \big( H(|u|^2)u\big) ,\quad t\in (0,T),\quad x\in \mathbb{R}~\text{or}~\mathbb{T},
\end{equation}
where $\alpha,\beta$ are real constants and $H$ is the Hilbert transformation.
We assume $\beta <0$ throughout this article.

In the periodic case, we proved in \cite{KT-gauge} that the Cauchy problem has a (forward-in-time) global solution for any initial data in $H^s(\mathbb{T})$ if $s>1/4$, with the solution map $u(0)=u_0 \mapsto u(\cdot )$ being (locally-in-time) continuous in the $H^s$ topology \emph{away from the origin} $u_0=0$.
More precisely, we proved the following claims:
\begin{enumerate}
\item[(i)] For any $s>1/4$ and any $R\geq r>0$, there exist $T>0$ and a solution map $u_0\mapsto u$ on the set $\{ u_0\in H^s(\mathbb{T}):\| u_0\|_{H^s}\leq R,\,\| u_0\|_{L^2}\geq r\}$ which gives a solution $u\in C([0,T];H^s(\mathbb{T}))$ to \eqref{kdnls} on $[0,T]$ with $u(0)=u_0$ and is continuous in the $H^s$ topology.
\item[(ii)] The above (non-trivial) solution $u(t)$ is smooth (especially in $H^1(\mathbb{T})$) for $t>0$, and then it extends to a global solution by means of the $H^1$-upper and $L^2$-lower \emph{a priori} bounds which are obtained for $H^1$ solutions of arbitrary size.
\end{enumerate}

\noindent
Note that the trivial solution $u\equiv 0$ is a global solution for $u_0=0$.
The continuity of the solution map at the origin can be verified if $s>1/2$ (\cite{KT22}), but it is open for $1/2\geq s>1/4$.
This is because {\it a priori} estimates and the local existence time given by the contraction argument depend on the reciprocal of the $L^2$ norm of solution for $1/2 \geq s > 1/4$.
In the non-periodic case, local well-posedness of the Cauchy problem in $H^s(\mathbb{R})$ can be proved for $s>3/2$ by the energy method, but no result seems to be currently available below $H^{3/2}$. 
To summarize, on $\mathbb{T}$ we have a global solution for $s>1/4$, while on $\mathbb{R}$ we only have a local solution for $s>3/2$.
We also note that these solutions to the Cauchy problem are unique in $C_tH^s_x$ if $s>3/2$.

The goal of this article is to prove an {\it a priori} $H^s$ estimate for small and smooth solutions to \eqref{kdnls} in the regularity range $1/2\geq s>1/4$.
In the periodic case, this and an approximation argument would imply the same estimate for the small (rough) $H^s$ solutions constructed in \cite{KT-gauge}, thus verifying the continuity of the solution map at the origin. 
Although our argument in the present paper is applicable to both periodic and non-periodic problems, we will mainly consider the periodic case, which seems technically more complicated.
(See Remark~\ref{rem:nonperiodic} below for a comment on the non-periodic case.)
\begin{theorem}\label{thm:main}
Let $\mathcal{M}=\mathbb{R}$ or $\mathbb{T}$ and $s>1/4$.
Then, there exist $\delta >0$ and $C>0$ such that if $0<T\leq 1$ and $u\in C([0,T];H^\infty (\mathcal{M}))$ is a smooth solution to \eqref{kdnls} on $\mathcal{M}$ satisfying $\| u(0)\|_{H^s}\leq \delta$, then it holds that 
\begin{equation}\label{aprioribound}
\| u \|_{L^\infty ([0,T];H^s)} \leq C\| u(0)\|_{H^s}.
\end{equation}
\end{theorem}

To establish the $H^s$ {\it a priori} bound \eqref{aprioribound}, we shall employ the short-time Fourier restriction method.
The short-time $X^{s,b}$ norms were introduced by Ionescu, Kenig, and Tataru \cite{IKT08}; the idea is to combine the $X^{s,b}$ analysis implemented in frequency-dependent small time intervals with an energy-type argument recovering the estimate on the whole interval.
The method has been applied to the modified Benjamin-Ono and the derivative NLS equations by Guo \cite{G11} in the non-periodic case, and in the periodic case by Schippa \cite{S21}, who used the $U^p$-$V^p$ type spaces instead of $X^{s,b}$.
The $X^{s,b}$ type spaces are suitable for detailed analysis on the resonance structure, while the $U^p$-$V^p$ type spaces work well with sharp cut-off functions in time.
For our purpose, the $U^p$-$V^p$ type spaces seem to be more convenient.
In our argument with the short-time Fourier restriction method, the modified energy plays a crucial role.
Our way of constructing the modified energy is slightly different from that in \cite{KT07}, \cite{G11} and \cite{S21} because of the presence of the Hilbert transformation in the cubic nonlinearity.
To be specific, \eqref{kdnls} has less symmetry than the DNLS, the cubic NLS and the modified Benjamin-Ono equations.
Moreover, it is known that the kinetic term $\beta \partial_x\big( H(|u|^2)u\big)$ in \eqref{kdnls} exhibits a kind of dissipation when $\beta <0$ (e.g., we have $\partial_t\| u(t)\|_{L^2}^2\leq 0$ for (smooth) solutions of \eqref{kdnls}, while the $L^2$ norm is conserved for the DNLS equation).
This dissipative nature has to be taken into account in the construction and the estimate of the modified energy, since otherwise there would remain some uncanceled terms with higher order derivatives compared to the corresponding estimate for the nonlinearity $\alpha \partial_x(|u|^2u)$.
Here, we do not have to estimate the difference of two solutions, since we only consider the continuity of the solution map at the origin.
So, we do not have to consider the modified energy for the difference of two solutions, either.

\begin{remark}
(i) In the case of DNLS, a similar \emph{a priori} $H^s$ estimate was obtained in \cite{G11,S21} for solutions of arbitrary size by using a rescaling argument.
Although the same idea may work for our problem \eqref{kdnls} to remove the smallness condition in Theorem~\ref{thm:main}, we will focus on small solutions in order to keep the argument not too complicated, and also because of our particular interest in the continuity of the solution map at the origin.

(ii)
An adaptation of the theory of low-regularity conservation laws for integrable PDEs by Killip, Vi\c{s}an, and Zhang \cite{KVZ18} might be another possible approach.
For the derivative NLS on $\mathbb{R}$ and on $\mathbb{T}$, the $H^s$ {\it a priori} estimate for $s>0$ was established in \cite{KSp} by this method.
Of course, KDNLS is not known to be completely integrable, but the method seems also useful to some dissipative perturbations of the integrable dispersive equations (e.g., the KdV-Burgers equation).
Unfortunately, this approach has not been successful for KDNLS up to now.
\end{remark}

The plan of the present paper is as follows.
In Section 2, we describe the definition of function spaces we work with, the short-time Strichartz estimates and the short-time bilinear Strichartz estimates.
Assuming the trilinear estimates and the modified energy estimate which are proved in later sections, we give the proof of our main Theorem \ref{thm:main}.
In Section 3, we give the trilinear estimate on the cubic nonlinearity in terms of short-time norms.
In Section 4, we define the modified energy and prove its estimates which are helpful for the short-time argument.

We would like to conclude this section with a couple of comments on Jean Ginibre's work about nonlinear wave and dispersive equations.
Ginibre started to study the scattering theory in the finite energy class for nonlinear Klein-Gordon and Schr\"odinger equations in late 1970's with Giorgio Velo.
Since then, he has made the great contribution to nonlinear partial differential equations, specifically, nonlinear wave and dispersive equations.
In early 1990's, Bourgain presented the so-called Fourier restriction method to study the well-posedness of the Cauchy problem for nonlinear dispersive equations such as nonlinear Schr\"odinger equations and the KdV equation.
The Fourier restriction method is very powerful, but it is rather complicated.
In fact, Bourgain's papers were not very easy to read.
Many people hoped the readable exposition on Bourgain's work about the Fourier restriction method.
In 1996, Ginibre wrote the nice exposition \cite{Gin} on the Fourier restriction method, which contained several new and important observations, for example, the relation between the Fourier restriction norm and the interaction representation in quantum physics.
This helped the Fourier restriction method to prevail among the community of nonlinear wave and dispersive equations.


\section{Function spaces, Strichartz estimates}\label{sec:2}

\subsection{Definition of function spaces}

\choice%
{%
Let us first recall the definition of $U^p$ and $V^p$ spaces.
\begin{definition}[cf.~\cite{HHK09,HHK09er} and \cite{KTVbook}]\label{def:UpVp}
Let $-\infty \leq a<b\leq \infty$, $I:=(a,b)$ and $1\leq p<\infty$.
\begin{itemize}
\item Let $\mathcal{P}(I)$ denote the set of all partitions of the interval $I$:
\[ \mathcal{P}(I):=\big\{ \{ t_k\}_{k=1}^K : K\in \mathbb{N},\,a<t_1<t_2<\dots <t_K<b\big\} .\]
\item A function $a:I\to L^2_x$ is called a $U^p(I)$-atom if 
\[ a(t)=\sum_{k=1}^K\phi_k\chi_{[t_k,t_{k+1})}(t);\quad \{ t_k\} _{k=1}^K\in \mathcal{P}(I),~\{ \phi _k\} _{k=1}^K\subset L^2_x~~\text{s.t.}\sum_{k=1}^K\| \phi _k\|_{L^2}^p=1,\]
where $\chi_\Omega$ denotes the characteristic function of a set $\Omega$ and we set $t_{K+1}:=b$.
The space $U^p(I)$ and its norm are defined as follows:
\begin{gather*}
U^p(I):=\Big\{ u:I\to L^2_x\,\Big|\,u=\sum_{j=1}^\infty \lambda_ja_j,\,\{ \lambda_j\} \in \ell ^1(\mathbb{N};\mathbb{C}),\,\text{$\{ a_j\}$:\,$U^p(I)$-atoms}\Big\} ,\\ 
\| u\|_{U^p(I)}:=\inf \Big\{ \sum _{j=1}^\infty |\lambda _j|\,\Big|\,\{ \lambda_j\} \in \ell ^1,\,\text{$\{ a_j\}$:\,$U^p(I)$-atoms s.t.}~u=\sum_{j=1}^\infty \lambda_ja_j\Big\} .
\end{gather*}
\item The normed space $V^p(I)$ is defined by
\begin{gather*}
V^p(I):=\big\{ v:I\to L^2_x\,\big|\, \| v\|_{V^p(I)}<\infty \big\} ,\\
\| v\|_{V^p(I)}:=\sup _{\{ t_k\}_{k=1}^K\in \mathcal{P}(I)}\Big( \sum _{k=1}^{K-1}\| v(t_{k+1})-v(t_k)\|_{L^2}^p+\| v(t_K)\|_{L^2}^p\Big) ^{1/p}.
\end{gather*}
\end{itemize}
\end{definition}
}{%
For $1\leq p<\infty$ and an interval $I=(a,b)$, $-\infty \leq a<b\leq \infty$, let $U^p(I),V^p(I)$ be the ($L^2_x$-valued) $\ell^p$-atomic space and the space of functions of bounded $p$-variation, respectively, on $I$.
For the precise definition of these spaces, see \cite{HHK09} (also \cite{HHK09er}) and \cite{KTVbook}.
}%
Recall that $U^p(I)$, $V^p(I)$ are Banach spaces and their elements are bounded functions from $I$ to $L^2_x$ which have one-sided limits at every point in $[a,b]$.
Moreover, $u\in U^p(I)$ is right continuous and satisfies $\lim_{t\to a}u(t)=0$.
As usual, we write $V^p_{rc}(I):=\{ v\in V^p(I):\text{$v$ is right continuous}\}$ and $V^p_{-,rc}(I):=\{ v\in V^p_{rc}(I):\lim_{t\to a}v(t)=0\}$, which are closed subspaces of $V^p(I)$.
We have $U^p(I)\subset U^q(I)$ for $p<q$ with continuous inclusion, and similarly for $V^p(I)$, $V^p_{rc}(I)$, $V^p_{-,rc}(I)$.
Following \cite{CHT12} (see also \cite{KT07}), we consider the space $DU^p(I):=\{ u':u\in U^p(I)\}$, where the derivative is taken in the sense of $L^2_x$-valued distributions on $I$.
For each $f\in DU^p(I)$ there is a unique $u\in U^p(I)$ satisfying $f=u'$, and hence $DU^p(I)$ is a \mbox{Banach} space equipped with the norm $\| f\|_{DU^p(I)}=\| u\|_{U^p(I)}$.
Finally, we write $U^p_\Delta (I):=\{ v:e^{i(-\cdot )\partial_x^2}v(\cdot )\in U^p(I)\}$ with $\| v\|_{U^p_\Delta(I)}:=\| e^{i(-\cdot )\partial_x^2}v(\cdot )\|_{U^p(I)}$, and similarly for $V^p_{\Delta}(I)$, $V^p_{rc,\Delta}(I)$, $V^p_{-,rc,\Delta}(I)$, and $DU^p_\Delta (I)$.
Note that $DU^p_\Delta (I)=\{ (\partial_t-i\partial_x^2)u:u\in U^p_\Delta(I)\}$.

We collect some basic properties of these spaces.
\begin{lemma}\label{lem:UpVp}
Let $I=(a,b)$ be any interval.
\begin{enumerate}
\item[(i)] [Continuous embeddings] For any $1\leq p<q<\infty$, we have 
\[ U^p_\Delta(I)\hookrightarrow V^p_{-,rc,\Delta}(I)\hookrightarrow U^q_\Delta(I)\hookrightarrow L^\infty (I;L^2),\qquad V^p_{rc,\Delta}(I)\hookrightarrow L^\infty (I;L^2).\]

\item[(ii)] [Duality] For $1<p<\infty$, we have $L^1(I;L^2)\hookrightarrow DU^p_\Delta(I)$ and
\begin{align*}
\| f\|_{DU^p_\Delta(I)}&=\Big\| \int _a^te^{i(t-t')\partial_x^2}f(t')\,dt'\Big\|_{U^p_\Delta(I)}\\
&=\sup \Big\{ \Big| \int_I\int f\bar{v}\,dx\,dt\Big| : v\in V^{p'}_{rc,\Delta}(I),\, \| v\|_{V^{p'}_\Delta(I)}\leq 1\Big\} \\
&\lesssim \| f\|_{L^1(I;L^2)}
\end{align*}
for $f\in L^1(I;L^2)$.

\item[(iii)] [Extension] Let $n\geq 1$ and $T:(L^2_x)^n\to L^1_{loc,x}$ be an operator such that it is either linear or conjugate linear in each variable.
Let $1\leq p,q\leq \infty$, and assume that the map $(\phi_1,\dots ,\phi_n)\mapsto [t\mapsto T(e^{it\partial_x^2}\phi_1,\dots ,e^{it\partial_x^2}\phi_n)]$ is bounded from $(L^2_x)^n$ to $L^p_t(I;L^q_x)$:
\[ \| T(e^{it\partial_x^2}\phi_1,\dots ,e^{it\partial_x^2}\phi_n)\|_{L^p_t(I;L^q_x)}\leq A\prod_{j=1}^n\| \phi_j\|_{L^2_x} \]
for some $A>0$.
Then, $T$ can be regarded as a map from $(U^p_\Delta(\mathbb{R}))^n$ to $L^p_t(I;L^q_x)$ by $(u_1,\dots ,u_n)\mapsto [t\mapsto T(u_1(t),\dots ,u_n(t))]$, and it is bounded: 
\[ \| T(u_1,\dots ,u_n)\|_{L^p_t(I;L^q_x)}\leq A\prod_{j=1}^n\| u_j\|_{U^p_\Delta(\mathbb{R})}.\]
Here, $U^p_\Delta(\mathbb{R})$ is replaced by $L^\infty _t(\mathbb{R};L^2_x)$ if $p=\infty$.

\item[(iv)] [Interpolation] Let $1\leq p<q<\infty$, $E$ be a Banach space, and $T:U^q_\Delta(I)\to E$ be a bounded, linear or conjugate linear operator such that $\| T\|_{U^q_\Delta(I)\to E}\leq C_q$, $\| T\|_{U^p_\Delta(I)\to E}\leq C_p$ for some $0<C_p\leq C_q<\infty$.
Then, we have
\[ \| T\|_{V^p_{-,rc,\Delta}(I)\to E}\lesssim \Big( 1+\log \frac{C_q}{C_p}\Big) C_p.\]
\end{enumerate}
\end{lemma}

\begin{proof}
(i) See, e.g., \cite[Propositions~2.2, 2.4, Corollary~2.6]{HHK09}.

(ii) The first equality holds by definition.
If $f\in L^1(I,L^2)$, the function $t\mapsto \int_a^te^{-it'\partial_x^2}f(t')\,dt'\in L^2$ is absolutely continuous and of bounded variation on $\overline{I}$.
Then, the second equality follows, e.g., from \cite[Theorem~2.8, Propositions~2.9, 2.10]{HHK09}.
The last inequality follows from the H\"older inequality and the embedding $V^{p'}_{rc,\Delta}(I)\hookrightarrow L^\infty(I;L^2)$.

(iii) See, e.g., \cite[Proposition~2.19]{HHK09}.

(iv) See, e.g., \cite[Proposition~2.20]{HHK09}.
\end{proof}

\choice%
{The following property of the $U^p$ space was stated in \cite[in the paragraph preceding Lemma~2.1]{CHT12}.
We give a precise statement and a proof for the reader's convenience.
\begin{lemma}\label{lem:Up-restriction}
Let $I=(a,b)$, $I'=(a',b')$ with $-\infty \leq a<b\leq \infty$ and $a<a'<b'\leq b$.
If $u\in U^p(I)$ and $u$ vanishes outside $I'$, then $u|_{I'}\in U^p(I')$ and $\| u|_{I'}\|_{U^p(I')}=\| u\|_{U^p(I)}$.
\end{lemma}

\begin{proof}
Let $\varepsilon>0$, and write $u$ as $u=\sum _j\lambda _ju_j$ with $\{ u_j\}$ being $U^p(I)$-atoms, $\{ \lambda_j\}\subset \mathbb{C}$ such that $\sum_j|\lambda_j|<\| u\|_{U^p(I)}+\varepsilon$.
Since $u$ vanishes outside $I'$, it holds that $u=\sum_j\lambda_j\chi_{[a',b')}u_j$. 
A simple argument shows that $\mu_j\chi_{[a',b')}u_j$ is a $U^p(I)$-atom for some $\mu _j\geq 1$ unless $\chi_{[a',b')}u_j\equiv 0$.
Hence, replacing $\lambda_j$ by $\lambda_j/\mu_j$ and excluding the terms with $\chi_{[a',b')}u_j\equiv 0$, we may assume that each $U^p(I)$-atom $u_j$ in the representation $u=\sum _j\lambda _ju_j$ is zero outside $[a',b')$.

To show $u|_{I'}\in U^p(I')$, we define a sequence $\{ u^{[J]}\}_{J\geq 1}$ by 
\[ u^{[J]}:=\sum_{j=1}^J\lambda_j\chi_{[t_J,b')}u_j,\]
where $\{ t_J\}_{J\geq 1}\subset (a',(a'+b')/2)$ is a decreasing sequence such that $\lim_{J\to \infty}t_J=a'$ and $u_j$ is constant on $[a',a'+2(t_J-a'))$ for all $J\geq j\geq 1$.
We observe that $\chi_{[t_J,b')}u_j|_{I'}$ is a $U^p(I')$-atom for each $J\geq j\geq 1$, and hence $u^{[J]}|_{I'}\in U^p(I')$ and $\| u^{[J]}|_{I'}\|_{U^p(I')}\leq \sum _{j=1}^{J}|\lambda_j|$.
Since the sum $\sum_j\lambda_ju_j$ is absolutely convergent in $U^p(I)\hookrightarrow L^\infty (I;L^2)$, we have the pointwise convergence $u^{[J]}(t)\to u(t)$ on $I'$.
For $J_1>J_2\geq 1$, we have
\begin{align*}
u^{[J_1]}-u^{[J_2]}&= \sum _{j=J_2+1}^{J_1}\lambda_j\chi_{[t_{J_1},b')}u_j+\sum _{j=1}^{J_2}\lambda_j\chi _{[t_{J_1},t_{J_2})}u_j \\
&=\sum _{j=J_2+1}^{J_1}\lambda_j\chi_{[t_{J_1},b')}u_j+\Big( \sum _{j=1}^{J_2}\lambda_ju_j(a')\Big) \chi_{[t_{J_1},t_{J_2})},
\end{align*}
so that (since $\sum_{j=1}^\infty \lambda_ju_j(a')=u(a')=0$)
\[ \| u^{[J_1]}|_{I'}-u^{[J_2]}|_{I'}\| _{U^p(I')}\leq \sum _{j=J_2+1}^\infty |\lambda_j|+\Big\| \sum_{j=1}^{J_2}\lambda_ju_j(a')\Big\|_{L^2}\to 0\qquad (J_2\to \infty ).\]
This shows that $\{ u^{[J]}|_{I'}\} \subset U^p(I')$ is a Cauchy sequence.
Since $U^p(I')$ is a Banach space, the limit $u|_{I'}$ belongs to $U^p(I')$.
Furthermore, we have 
\[ \| u|_{I'}\|_{U^p(I')}=\lim_{J\to \infty}\| u^{[J]}|_{I'}\|_{U^p(I')}\leq \sum_j|\lambda_j|< \| u\|_{U^p(I)}+\varepsilon ,\]
which implies the inequality $\| u|_{I'}\|_{U^p(I')}\leq \| u\|_{U^p(I)}$.

The reverse inequality is immediate, since every $U^p(I')$-atom gives a $U^p(I)$-atom by the zero extension.
\end{proof}}%
{%
}%

Now, we define the short-time norms.
In this article, we use capital letters $N,K,\dots $ for dyadic integers $1,2,4,8,\dots$.

\begin{definition}\label{defn:spaces}
First of all, we fix a bump function 
\choice{%
\[ \eta \in C^\infty_0(\mathbb{R})\quad \text{even, monotone on $[0,\infty)$, and $\chi_{[-4/3,4/3]}\leq \eta \leq \chi_{[-5/3,5/3]}$.}\]
}{%
\[ \eta \in C^\infty_0(\mathbb{R})\quad \text{even, monotone on $[0,\infty)$, and $\chi_{[-4/3,4/3]}\leq \eta \leq \chi_{[-5/3,5/3]}$,}\]
where $\chi_\Omega$ denotes the characteristic function of a set $\Omega$.
}%
Define $\{ \psi _N\}_{N\geq 1} \subset C^\infty_0(\mathbb{R})$ by 
\[ \psi_1(\xi):=\eta (\xi),\qquad \psi _N(\xi):=\eta (\xi /N)-\eta (2\xi/N)\quad \text{for $N\geq 2$}, \]
so that $1=\sum _{N\geq 1}\psi _N(\xi)$ and $\mathrm{supp}\,(\psi_N)\subset \mathcal{I}_N$, where 
\[ \mathcal{I}_1:=[-2,2],\qquad \mathcal{I}_N:=[-2N,2N]\setminus (-N/2,N/2)\quad \text{for $N\geq 2$.}\]
We define the corresponding Littlewood-Paley projections $P_N:=\mathcal{F}^{-1}_\xi \psi _N\mathcal{F}_x$.

Next, we define frequency-localized short-time norms $F_N(T)$, $G_N(T)$ for functions $u:[0,T]\to L^2$ with $\mathrm{supp}\,(\hat{u}(t,\xi))\subset [0,T]\times \mathcal{I}_N$ by
\begin{align*}
\| u\|_{F_N(T)}&:=\sup_{I=[a,b)\subset [0,T],\,|I|\leq N^{-1}}\| \chi_Iu\| _{U^2_\Delta (\mathbb{R})},\\
\| u\|_{G_N(T)}&:=\sup_{I=(a,b)\subset [0,T],\,|I|\leq N^{-1}}\| u|_I\| _{DU^2_\Delta (I)}.
\end{align*}
In the definition of the $F_N(T)$ norm, we regard $\chi_Iu$ as a function on $\mathbb{R}$ by extending it by zero outside $I$.
Here, we consider half-open intervals $I=[a,b)$ so that $\chi_Iu$ can be right continuous, and we avoid writing the norm $\| \chi_Iu\|_{U^2_\Delta(\mathbb{R})}$ as $\| u\|_{U^2_\Delta ((a,b))}$ since the $U^2_\Delta((a,b))$ norm can be defined only for functions satisfying $\lim_{t\to a+0}u(t)=0$.
The short-time $U^2_\Delta$-type space on $[0,T]$ with spatial regularity $s\in \mathbb{R}$ is defined by
\[ F^s(T):=\big\{ u\in C([0,T];H^s): \| u\|_{F^s(T)}:=\big\| N^s\| P_Nu\|_{F_N(T)}\big\|_{\ell ^2_N}<\infty \big\} .\]
To measure the nonlinearity, the following short-time norm is used:
\[ \| u\|_{G^s(T)}:=\big\| N^s\| P_Nu\|_{G_N(T)}\big\|_{\ell ^2_N}. \]
We also need the following energy norm:
\[ \| u\|_{E^s(T)}:=\big\| N^s\| P_Nu\|_{L^\infty ([0,T];L^2)}\big\|_{\ell^2_N}.\]
\end{definition}

\subsection{Proof of the main theorem}
It is known 
\choice%
{(e.g., \cite[Lemma~3.1]{CHT12}, where the fact shown in Lemma~\ref{lem:Up-restriction} is needed) }%
{(e.g., \cite[Lemma~3.1]{CHT12}) }%
that the norms defined above satisfy the basic linear estimate
\[ \| u\| _{F^s(T)}\lesssim \| u\|_{E^s(T)}+\big\| (\partial_t-i\partial_x^2)u\big\|_{G^s(T)}\]
for any $s\in \mathbb{R}$. 
Then, what we need to show are the trilinear estimate
\[ \big\| \partial_x(|u|^2u)\big\|_{G^s(T)}+\big\| \partial_x\big( H(|u|^2)u\big) \big\|_{G^s(T)}\lesssim \| u\|_{F^s(T)}^3\]
and the energy estimate
\[ \| u\| _{E^s(T)}\lesssim \| u(0)\|_{H^s}+\| u\|_{F^s(T)}^3.\]
We will prove the trilinear estimate for general functions $u\in F^s(T)$ in Section~\ref{sec:trilinear}, and the energy estimate for smooth solutions of \eqref{kdnls} with small initial data in Section~\ref{sec:energy}. 
Both of these estimates require $s>1/4$, and also have the constants uniform for $T\in (0,1]$ but growing for $T>1$.

Let us admit these estimates and prove Theorem~\ref{thm:main}.
For $0<T'\leq T$ and a smooth solution $u$ with initial data small in $H^s$, define
\[ X_s(T'):=\| u\|_{E^s(T')}+\big\| \partial_x(|u|^2u)\big\|_{G^s(T')}+\big\| \partial_x\big( H(|u|^2)u\big) \big\|_{G^s(T')}.\]
The above three estimates show that
\[ X_s(T')\lesssim \| u(0)\|_{H^s}+X_s(T')^3.\]
On the other hand, it is easy to show (e.g., for $u\in C([0,T];H^{s+1+})$) that $X_s(T')$ is continuous in $T'$ and 
\[ \limsup_{T'\to +0}X_s(T')\lesssim \| u(0)\|_{H^s}.\]
Hence, by a bootstrap argument, we have
\[ X_s(T')\lesssim \| u(0)\|_{H^s},\qquad 0<T'\leq T.\]
Since $\| u\|_{L^\infty ([0,T];H^s)}\lesssim \| u\|_{E^s(T)}$, this concludes the proof of Theorem~\ref{thm:main}.

\subsection{Short-time $L^6$ and bilinear Strichartz estimates}

Most of Strichartz-type estimates for the non-periodic Schr\"odinger equation are known to hold for the periodic problem in the short-time setting, and these estimates will be used as basic tools to prove the trilinear and energy estimates.
We begin with the following $L^6$ Strichartz estimate.
\begin{lemma}\label{lem:L6}
For $N\geq 1$ and $0<\delta \lesssim N^{-1}$, we have
\[ \big\| P_{\leq N}e^{it\partial_x^2}\phi \big\|_{L^6([0,\delta ];L^6(\mathbb{T}))}\lesssim \| \phi \|_{L^2}.\]
As a consequence, we have
\[ \| P_{\leq N}u\|_{L^6(I;L^6(\mathbb{T}))}\lesssim \| \chi_Iu\|_{U^6_\Delta (\mathbb{R})}\]
for any interval $I=[a,b)\subset \mathbb{R}$ with $|I|\lesssim N^{-1}$ and any $u:I\to L^2$ such that the zero extension $\chi_Iu$ belongs to $U^6_\Delta (\mathbb{R})$.
\end{lemma}

\begin{proof}
The first estimate was shown in \cite[Proposition~2.9]{BGT04}.
To obtain the second claim, we use Lemma~\ref{lem:UpVp}~(iii) with the operator $T:\phi \mapsto P_{\leq N}\phi$ and apply the resulting estimate to $\chi_Iu\in U^6_\Delta(\mathbb{R})$.
\end{proof}

\choice%
{Actually, the estimate of this type holds on $L^p([0,\delta ];L^q(\mathbb{T}))$ for any admissible pair $(p,q)$; $2/p+1/q=1/2$, $2\leq q\leq \infty$.
This was proved in \cite{BGT04} on more general compact manifolds, by applying Keel-Tao's abstract theory to the short-time dispersive estimate: 
\begin{equation}\label{est:disp}
\big\| P_{\leq N}e^{it\partial_x^2}\phi \big\|_{L^\infty (\mathbb{T})}\lesssim |t|^{-1/2}\| \phi\|_{L^1(\mathbb{T})},\qquad 0<|t|\leq N^{-1}.
\end{equation}
The proof of \eqref{est:disp} for general compact manifolds in \cite{BGT04} was based on semiclassical calculus.
For the case of $\mathbb{T}$, there is a more direct proof, which we would like to present below.
(We learned this proof from a recent paper \cite{LPRT19} studying the Zakharov-Kuznetsov equation on $\mathbb{T}^2$, but the argument would have been well known since a long time ago.)

\begin{proof}[Proof of \eqref{est:disp}]
The claim is reduced to the following pointwise bound on the kernel of $P_{\leq N}e^{it\partial_x^2}$:
\[ \Big| \sum _{\xi \in \mathbb{Z}}\eta _N(\xi )e^{ix\xi -it\xi ^2}\Big| \lesssim |t|^{-1/2},\qquad 0<|t|\leq N^{-1},\quad x\in \mathbb{T}, \]
where $\eta _N(\cdot )=\eta (\cdot /N)=\sum _{K\leq N}\psi _K(\cdot )$ is the symbol of the operator $P_{\leq N}$. 
By the Poisson summation formula $\sum _{\xi \in \mathbb{Z}}f(\xi ) =\sum _{n\in \mathbb{Z}}\hat{f}(2\pi n)$ ($f\in \mathcal{S}(\mathbb{R})$), it suffices to prove
\begin{equation}\label{est:kernel}
\Big| \sum _{n\in \mathbb{Z}}F_{N,t}(x-2\pi n) \Big| \lesssim |t|^{-1/2},\qquad 0<|t|\leq N^{-1},\quad |x|\leq \pi,
\end{equation}
where 
\[ F_{N,t}(y):=\int _{\mathbb{R}}\eta _N(\xi )e^{iy\xi -it\xi ^2}\,d\xi ,\qquad y\in \mathbb{R}.\]
We set $\Phi (\xi ):=y\xi -t\xi ^2$.
On one hand, $|\Phi''|\equiv 2|t|$ and the van~der~Corput lemma show that $|F_{N,t}(y)|\lesssim |t|^{-1/2}$ uniformly in $y\in \mathbb{R}$.
On the other hand, for $|y|\geq 5$, the restrictions $|\xi |\leq 2N$ and $|t|\leq N^{-1}$ imply that $\Phi (\xi )$ has no stationary point and $|\Phi'(\xi )|\sim |y|$, $|\Phi''(\xi )|=2|t|\lesssim N^{-1}$.
Hence, after doing integration by parts twice:
\begin{align*}
|F_{N,t}(y)|&=\Big| \int _{\mathbb{R}}\Big\{ \frac{1}{i\Phi'}\Big( \frac{1}{i\Phi'}\eta_N\Big)'\Big\}'e^{i\Phi}\,d\xi \Big| \\
&\lesssim \int _{-2N}^{2N}\Big( \frac{1}{|\Phi'|^2}|\eta_N''|+\frac{|\Phi''|}{|\Phi'|^3}|\eta_N'|+\frac{|\Phi''|^2}{|\Phi'|^4}|\eta_N|\Big) \,d\xi,
\end{align*}
we see that $|F_{N,t}(y)|\lesssim N^{-1}|y|^{-2}$.
From these estimates, for $0<|t|\leq N^{-1}$ and $|x|\leq \pi$ we have
\[ \Big| \sum _{n\in \mathbb{Z}}F_{N,t}(x-2\pi n) \Big| \lesssim N^{-1}\sum _{|n|\geq 2}\frac{1}{(x-2\pi n)^2} +\sum _{n=-1}^1|t|^{-1/2}\lesssim |t|^{-1/2}, \]
which verifies \eqref{est:kernel} and thus \eqref{est:disp}.
\end{proof}}%
{%
}%

As a counterpart of the bilinear Strichartz estimate of Ozawa and Tsutsumi \cite[Theorem~2~(1)]{OT98}, we have the following short-time bilinear Strichartz estimate on $\mathbb{T}$.
A Fourier analytic proof was given in \cite{MV08}, which we will recall below for completeness.
\begin{lemma}\label{lem:bs0}
For $K\geq 1$ and $\delta >0$, we have 
\begin{equation}\label{est:bs0+}
\big\| P_K\big( e^{it\partial_x^2}\phi_1\overline{e^{it\partial_x^2}\phi_2}\big) \big\|_{L^2([0,\delta ];L^2(\mathbb{T}))}\lesssim \Big( \frac{1+K\delta}{K}\Big) ^{1/2}\| \phi_1\|_{L^2}\| \phi _2\|_{L^2}.
\end{equation}
In particular, if $N_1\gg N_2\geq 1$ and $\phi_1,\phi_2$ satisfy $\mathrm{supp}\,(\hat{\phi}_j)\subset \mathcal{I}_{N_j}$, then for $0<\delta \lesssim N_1^{-1}$ we have
\begin{equation}\label{est:bs0}
\big\| e^{it\partial_x^2}\phi_1\overline{e^{it\partial_x^2}\phi_2}\big\|_{L^2([0,\delta ];L^2(\mathbb{T}))}\lesssim N_1^{-1/2}\| \phi_1\|_{L^2}\| \phi _2\|_{L^2}.
\end{equation}
\end{lemma}

\begin{remark}\label{rem:bs}
The latter estimate \eqref{est:bs0} clearly holds regardless of the complex conjugation, while this is not the case for the former estimate \eqref{est:bs0+}.
For the product without conjugation of two functions of comparable frequencies, we can deduce, for instance, the following result from \eqref{est:bs0+}:
if $|\xi_1-\xi_2|\sim K$ for any $\xi_j\in \mathrm{supp}\,(\hat{\phi}_j)$, $j=1,2$, then 
\[ \big\| e^{it\partial_x^2}\phi_1e^{it\partial_x^2}\phi_2 \big\|_{L^2([0,\delta ];L^2(\mathbb{T}))}\lesssim \Big( \frac{1+K\delta}{K}\Big) ^{1/2}\| \phi_1\|_{L^2}\| \phi _2\|_{L^2}.\]
\choice%
{In particular, if $N_1\sim N_2$, $\mathrm{supp}\,(\hat{\phi}_1)\subset \mathcal{I}_{N_1}\cap \mathbb{R}_+$ and $\mathrm{supp}\,(\hat{\phi}_2)\subset \mathcal{I}_{N_2}\cap \mathbb{R}_-$, then for $0<\delta \lesssim N_1^{-1}$ we have
\[ \big\| e^{it\partial_x^2}\phi_1e^{it\partial_x^2}\phi_2 \big\|_{L^2([0,\delta ];L^2(\mathbb{T}))}\lesssim N_1^{-1/2}\| \phi_1\|_{L^2}\| \phi _2\|_{L^2}.\]}%
{%
}%
\end{remark}

\begin{remark}\label{rem:V2}
As for the $L^6$ Strichartz estimate, from \eqref{est:bs0+} and Lemma~\ref{lem:UpVp} (iii) we immediately obtain the corresponding bilinear estimates in $U^2_\Delta$: for $I=[a,b)$ with $|I|\lesssim K^{-1}$ we have
\[ \| P_K(u\bar{v})\|_{L^2(I;L^2(\mathbb{T}))}\lesssim K^{-1/2}\| \chi_Iu\|_{U^2_\Delta (\mathbb{R})}\| \chi_Iv\| _{U^2_\Delta (\mathbb{R})}.\]
A similar extension is valid also for Lemmas~\ref{lem:bs1} and \ref{lem:bs2} below.
On the other hand, by the Bernstein and H\"older inequalities and the assumption $|I|\lesssim K^{-1}$, together with the embedding $U^p_\Delta \hookrightarrow L^\infty L^2$, we have
\[ \| P_K(u\bar{v})\|_{L^2(I;L^2(\mathbb{T}))}\lesssim K^{1/2}|I|^{1/2}\| u\bar{v}\|_{L^\infty(\mathbb{R};L^1(\mathbb{T}))}\lesssim \| u\|_{U^2_\Delta(\mathbb{R})}\| v\|_{U^4_\Delta(\mathbb{R})}.\]
By applying Lemma~\ref{lem:UpVp}~(iv) to the operator $v\mapsto P_K(u\bar{v})$ with these estimates, we have
\[ \| P_K(u\bar{v})\|_{L^2(I;L^2(\mathbb{T}))}\lesssim K^{-1/2}(1+\log K)\| \chi_Iu\|_{U^2_\Delta (\mathbb{R})}\| \chi_Iv\| _{V^2_\Delta (\mathbb{R})}\]
for $u,v:I\to L^2$ such that $\chi_Iu\in U^2_\Delta(\mathbb{R})$ and $\chi_Iv\in V^2_{-,rc,\Delta}(\mathbb{R})$.
\end{remark}

\begin{proof}[Proof of Lemma~\ref{lem:bs0}]
If $K=1$, the claim follows from the H\"older inequality in $t$ and the Bernstein inequality in $x$.

Assume $K>1$.
We observe that
\begin{align*}
P_K\big( e^{it\partial_x^2}\phi_1\overline{e^{it\partial_x^2}\phi_2}\big) &=\sum _{n_1,n_2}e^{i(n_1-n_2)x}e^{i(-n_1^2+n_2^2)t}\psi_K(n_1-n_2)\hat{\phi}_1(n_1)\overline{\hat{\phi}_2(n_2)}\\
&=\sum _ne^{inx}\psi_K(n)e^{-in^2t}\sum _{n_2}\hat{\phi}_1(n+n_2)\overline{\hat{\phi}_2(n_2)}e^{-2inn_2t}.
\end{align*}
By the Plancherel theorem and the change of variable $t'=-2nt$, we have
\begin{align*}
&\big\| P_K\big( e^{it\partial_x^2}\phi_1\overline{e^{it\partial_x^2}\phi_2}\big) \big\|_{L^2([0,\delta ];L^2(\mathbb{T}))}^2\\
&=\int _0^\delta \sum _n|\psi_K(n)|^2 \Big| \sum _{n_2}\hat{\phi}_1(n+n_2)\overline{\hat{\phi}_2(n_2)}e^{-2inn_2t}\Big| ^2\,dt\\
&=\sum _n|\psi_K(n)|^2 \frac{1}{-2n}\int _0^{-2n\delta}\Big| \sum _{n_2}\hat{\phi}_1(n+n_2)\overline{\hat{\phi}_2(n_2)}e^{in_2t'}\Big| ^2\,dt'\\
&\lesssim \sum _n\frac{1+K\delta}{K}\int _0^{2\pi}\Big| \sum _{n_2}\hat{\phi}_1(n+n_2)\overline{\hat{\phi}_2(n_2)}e^{in_2t'}\Big| ^2\,dt'.
\end{align*}
Since the last term is equal to $\frac{1+K\delta}{K}\| \phi_1\|_{L^2}^2\| \phi _2\|_{L^2}^2$ by the Plancherel theorem, the claimed estimate follows.
\end{proof}

In order to deal with the nonlinearity of \eqref{kdnls} including the Hilbert transformation, we prepare the next two lemmas.
These estimates can be shown in the same manner as Lemma~\ref{lem:bs0}.
\begin{lemma}\label{lem:bs1}
Let $\phi _1,\phi _2,\phi _3\in L^2(\mathbb{T})$ satisfy $\mathrm{supp}\,(\hat{\phi}_j)\subset \mathcal{I}_{N_j}$, and assume that $N_1\gg N_2,N_3$.
Then, for $0<\delta \lesssim N_1^{-1}$ we have 
\[ \big\| e^{it\partial_x^2}\phi_1\cdot H\big( e^{it\partial_x^2}\phi_2\overline{e^{it\partial_x^2}\phi_3}\big) \big\|_{L^2([0,\delta ];L^2(\mathbb{T}))}\lesssim \Big( \frac{N_2\wedge N_3}{N_1}\Big) ^{1/2}\| \phi_1\|_{L^2}\| \phi _2\|_{L^2}\| \phi _3\|_{L^2}.\]
The same estimate holds if $e^{it\partial_x^2}\phi_1$ is replaced by $\overline{e^{it\partial_x^2}\phi_1}$, and also if $H$ is replaced by any Fourier multiplier with bounded symbol (such as $P_{\leq N}$).
\end{lemma}

\begin{proof}
Since $\overline{H(u\bar{v})}=H(\bar{u}v)$, we may assume $N_2\leq N_3$.
We observe that
\begin{align*}
&e^{it\partial_x^2}\phi_1\cdot H\big( e^{it\partial_x^2}\phi_2\overline{e^{it\partial_x^2}\phi_3}\big) \\
&=\sum _{n_1,n_2,n_3}e^{i(n_1+n_2-n_3)x}e^{i(-n_1^2-n_2^2+n_3^2)t}(-i)\mathrm{sgn}(n_2-n_3)\hat{\phi}_1(n_1)\hat{\phi}_2(n_2)\overline{\hat{\phi}_3(n_3)}\\
&=-i\sum _ne^{inx}\sum _{n_2}e^{-i(n_2^2+(n-n_2)^2)t}\hat{\phi}_2(n_2)\sum _{n_3}\mathrm{sgn}(n_2-n_3)\hat{\phi}_1(n-n_2+n_3)\overline{\hat{\phi}_3(n_3)}e^{-2in_3(n-n_2)t},
\end{align*}
and hence
\begin{align*}
&\big\| e^{it\partial_x^2}\phi_1\cdot H\big( e^{it\partial_x^2}\phi_2\overline{e^{it\partial_x^2}\phi_3}\big) \big\|_{L^2([0,\delta ];L^2(\mathbb{T}))}^2\\
&=\int _0^\delta \sum _n\bigg| \sum _{n_2}e^{-i(n_2^2+(n-n_2)^2)t}\hat{\phi}_2(n_2)\sum _{n_3}\mathrm{sgn}(n_2-n_3)\hat{\phi}_1(n-n_2+n_3)\overline{\hat{\phi}_3(n_3)}e^{-2in_3(n-n_2)t}\bigg| ^2\,dt\\
&\leq \int _0^\delta \sum _n\bigg[ \sum _{n_2}|\hat{\phi}_2(n_2)|\Big| \sum _{n_3}\mathrm{sgn}(n_2-n_3)\hat{\phi}_1(n-n_2+n_3)\overline{\hat{\phi}_3(n_3)}e^{-2in_3(n-n_2)t}\Big| \bigg] ^2\,dt\\
&\leq \bigg[ \sum _{n_2}|\hat{\phi}_2(n_2)|\bigg( \sum _n\int _0^\delta \Big| \sum _{n_3}\mathrm{sgn}(n_2-n_3)\hat{\phi}_1(n-n_2+n_3)\overline{\hat{\phi}_3(n_3)}e^{-2in_3(n-n_2)t}\Big| ^2\,dt\bigg) ^{1/2}\bigg] ^2,
\end{align*}
where at the last step we have used the Minkowski inequality to replace the $L^2_t\ell^2_n\ell^1_{n_2}$ norm with $\ell^1_{n_2}\ell^2_nL^2_t$.
Now, for fixed $|n|\sim {N_1}$ and $n_2\in \mathcal{I}_{N_2}$, we have $|2(n-n_2)\delta|\lesssim 2\pi$, and thus
\begin{align*}
&\int _0^\delta \Big| \sum _{n_3}\mathrm{sgn}(n_2-n_3)\hat{\phi}_1(n-n_2+n_3)\overline{\hat{\phi}_3(n_3)}e^{-2in_3(n-n_2)t}\Big| ^2\,dt\\
&=\frac{1}{-2(n-n_2)}\int _0^{-2(n-n_2)\delta}\Big| \sum _{n_3}\mathrm{sgn}(n_2-n_3)\hat{\phi}_1(n-n_2+n_3)\overline{\hat{\phi}_3(n_3)}e^{in_3t'}\Big| ^2\,dt' \\
&\lesssim N_1^{-1}\int _0^{2\pi}\Big| \sum _{n_3}\mathrm{sgn}(n_2-n_3)\hat{\phi}_1(n-n_2+n_3)\overline{\hat{\phi}_3(n_3)}e^{in_3t'}\Big| ^2\,dt'\\
&\lesssim N_1^{-1}\sum _{n_3}\big| \hat{\phi}_1(n-n_2+n_3)\big| ^2\big| \hat{\phi}_3(n_3)\big| ^2.
\end{align*}
Hence, we have
\begin{align*}
&\big\| e^{it\partial_x^2}\phi_1\cdot H\big( e^{it\partial_x^2}\phi_2\overline{e^{it\partial_x^2}\phi_3}\big) \big\|_{L^2([0,\delta ];L^2(\mathbb{T}))}^2\\
&\lesssim \bigg[ \sum _{n_2}|\hat{\phi}_2(n_2)|\bigg( \sum _nN_1^{-1}\sum _{n_3}\big| \hat{\phi}_1(n-n_2+n_3)\big| ^2\big| \hat{\phi}_3(n_3)\big| ^2\bigg) ^{1/2}\bigg] ^2\\
&\lesssim N_2\sum _{n_2}|\hat{\phi}_2(n_2)|^2\sum _nN_1^{-1}\sum _{n_3}\big| \hat{\phi}_1(n-n_2+n_3)\big| ^2\big| \hat{\phi}_3(n_3)\big| ^2\\
&\lesssim N_1^{-1}N_2\| \phi_1\|_{L^2}^2\| \phi _2\|_{L^2}^2\| \phi _3\|_{L^2}^2,
\end{align*}
as desired.
\end{proof}

\begin{lemma}\label{lem:bs2}
Let $\phi _1,\phi _2,\phi _3\in L^2(\mathbb{T})$ satisfy $\mathrm{supp}\,(\hat{\phi}_j)\subset \mathcal{I}_{N_j}$, and assume that $N_1\sim N_2\gg N_3$.
Further, assume $K\ll N_1$.
Then, for $0<\delta \lesssim N_1^{-1}$ we have 
\[ \big\| HP_{\leq K}\big( e^{it\partial_x^2}\phi_1\overline{e^{it\partial_x^2}\phi_2}\big) e^{it\partial_x^2}\phi_3 \big\|_{L^2([0,\delta ];L^2(\mathbb{T}))}\lesssim \Big( \frac{K}{N_1}\Big) ^{1/2}\| \phi_1\|_{L^2}\| \phi _2\|_{L^2}\| \phi _3\|_{L^2}.\]
The same estimate holds if $e^{it\partial_x^2}\phi_3$ is replaced by $\overline{e^{it\partial_x^2}\phi_3}$, and also if $H$ is replaced by any Fourier multiplier with bounded symbol.
\end{lemma}

\begin{proof}
By an almost orthogonality argument, we can restrict the frequencies of $\phi_1$ and $\phi_2$ onto intervals of length $K$.
Then, the same argument as for the preceding lemma can be used.
\end{proof}

\begin{remark}\label{rem:nonperiodic}
We note that all the above short-time $L^6$ and bilinear Strichartz estimates (Lemmas~\ref{lem:L6}, \ref{lem:bs0}, \ref{lem:bs1}, and \ref{lem:bs2}) are true in the non-periodic case as well.
In fact, these estimates hold on $\mathbb{R}$ without restricting to a frequency-dependent short time interval (i.e., with the $L^6_{t,x}$ or $L^2_{t,x}$ norm over $\mathbb{R}\times \mathbb{R}$ on the left-hand side).
Concerning Lemmas~\ref{lem:bs1} and \ref{lem:bs2}, this can be shown by a slight modification of the proofs for the periodic estimates given above.
\end{remark}


\section{Trilinear estimate in the short-time norms}\label{sec:trilinear}

In this section, we shall prove the following trilinear estimate in the $G^s(T)$ norm.
\begin{proposition}\label{prop:trilinear}
For $s>1/4$ and $0<T\leq 1$, we have
\[ \big\| \partial_x(u_1\bar{u}_2u_3)\big\|_{G^s(T)}+\big\| \partial_x\big( H(u_1\bar{u}_2)u_3\big) \big\|_{G^s(T)}\lesssim \| u_1\|_{F^s(T)}\| u_2\|_{F^s(T)}\| u_3\|_{F^s(T)}. \]
\end{proposition}

\begin{proof}
We only consider the second term on the left-hand side with the Hilbert transformation.
The first term (for DNLS) was treated in \cite{G11,S21}; in fact, it can be dealt with in a similar manner but more easily.

We apply dyadic decompositions as
\[  H(u_1\bar{u}_2)u_3 = \sum _{\begin{smallmatrix} N_1,\dots ,N_4\geq 1 \\ N^*_1\sim N^*_2\end{smallmatrix}}P_{N_4}\big( H(P_{N_1}u_1P_{N_2}\bar{u}_2)P_{N_3}u_3\big) ,\]
where we write $N^*_1,\dots ,N^*_4$ to denote the numbers $N_1,\dots ,N_4$ rearranged in decreasing order.
It then suffices to show for each $\underline{N}=(N_1,\dots ,N_4)$ the localized estimate
\begin{equation}\label{est:tri-local}
\begin{aligned}
&\big\| \partial_xP_{N_4}\big( H(P_{N_1}u_1P_{N_2}\bar{u}_2)P_{N_3}u_3\big) \big\|_{G_{N_4}(T)}\\
&\quad \lesssim C(\underline{N})\| P_{N_1}u_1\|_{F_{N_1}(T)}\| P_{N_2}u_2\|_{F_{N_2}(T)}\| P_{N_3}u_3\|_{F_{N_3}(T)}
\end{aligned}
\end{equation}
with some $C(\underline{N})$ satisfying
\[ C(\underline{N})\lesssim \frac{N_1^sN_2^sN_3^s}{N_4^s}(N^*_3)^{0-}.\]
(Since $N^*_1\sim N^*_2$, the factor $(N^*_3)^{0-}$ allows us to restore the claimed estimate by summing up \eqref{est:tri-local} in $\underline{N}$.)
We will actually obtain \eqref{est:tri-local} with smaller $C(\underline{N})$ which satisfies
\begin{equation}\label{bound:c}
C(\underline{N})\lesssim \Big( \frac{N_1N_2N_3}{N^*_1}\Big) ^{(1/4)+}.
\end{equation}
From the definition of the $F_N(T)$, $G_N(T)$ norms, we need to prove
\begin{align*}
&\sup_{I_4=(a,b)\subset [0,T],\, |I_4|\leq N_4^{-1}}\big\| \partial_xP_{N_4}\big( H(P_{N_1}u_1P_{N_2}\bar{u}_2)P_{N_3}u_3\big) \big\|_{DU^2_\Delta (I_4)}\\
&\quad \lesssim C(\underline{N})\prod _{j=1}^3\sup_{I_j=[a,b)\subset [0,T],\, |I_j|\leq N_j^{-1}}\| \chi_{I_j}P_{N_j}u_j\|_{U^2_\Delta (\mathbb{R})}.
\end{align*}
Since $\partial_xP_{N_4}\big( H(P_{N_1}u_1P_{N_2}\bar{u}_2)P_{N_3}u_3\big) \in L^1([0,T];L^2)$ for $u_1,u_2,u_3\in F^s(T)\subset C([0,T];H^s)$, by Lemma~\ref{lem:UpVp} (ii) it suffices to prove either
\begin{equation}\label{est:tri1'}
\begin{aligned}
&\big\| \partial_xP_{N_4}\big( H(P_{N_1}u_1P_{N_2}\bar{u}_2)P_{N_3}u_3\big) \big\|_{L^1(I_4;L^2)}\\
&\quad \lesssim C(\underline{N})\prod _{j=1}^3\sup_{I_j=[a,b)\subset [0,T],\, |I_j|\leq N_j^{-1}}\| \chi_{I_j}P_{N_j}u_j\|_{U^2_\Delta (\mathbb{R})}
\end{aligned}
\end{equation}
or
\begin{equation}\label{est:tri2}
\begin{aligned}
&\Big| \int_{I_4} \int H(P_{N_1}u_1P_{N_2}\bar{u}_2)P_{N_3}u_3 \cdot \partial_xP_{N_4}\bar{u}_4\,dx\,dt\Big| \\
&\quad \lesssim C(\underline{N})\| u_4\|_{V^2_\Delta(I_4)}\prod _{j=1}^3\sup_{I_j=[a,b)\subset [0,T],\, |I_j|\leq N_j^{-1}}\| \chi_{I_j}P_{N_j}u_j\|_{U^2_\Delta (\mathbb{R})}
\end{aligned}
\end{equation}
for any $I_4=(a,b)\subset [0,T]$, $|I_4|\leq N_4^{-1}$ and any $u_4\in V^2_{rc,\Delta}(I_4)$.

When $N_4\ll N^*_1$, the time scale on the right-hand side is finer than that on the left-hand side, and therefore we need to first divide $I_4$ into sub-intervals of size $\leq (N^*_1)^{-1}$, the number of which is $O(N_1^*/N_4)$.
Then, to verify \eqref{est:tri1'} we need to show
\begin{equation}\label{est:tri1}
\begin{aligned}
&(N^*_1)^{1/2}\big\| P_{N_4}\big( H(P_{N_1}u_1P_{N_2}\bar{u}_2)P_{N_3}u_3\big) \big\|_{L^2(I;L^2)}\\
&\quad \lesssim C(\underline{N})\prod _{j=1}^3\| \chi_{I}P_{N_j}u_j\|_{U^2_\Delta (\mathbb{R})}
\end{aligned}
\end{equation}
for any interval $I$ with
\[ I=[a,b)\subset [0,T],\quad |I|\leq (N^*_1)^{-1}.\]
In fact, \eqref{est:tri1} implies \eqref{est:tri1'} by the Schwarz inequality in $t$ and the Bernstein inequality in $x$.
From now on, we write simply $u_j$ for $P_{N_j}u_j$.

\medskip
\noindent
\textbf{Case (I)} $N_4\sim N^*_1$.

(Ia) [high$\times$high$\times$high$\to$high]  $N_1\sim N_2\sim N_3\sim N_4$.\\
We simply use the $L^6$ Strichartz estimate (Lemma~\ref{lem:L6}) for each function:
\begin{align*}
(N^*_1)^{1/2}\| H(u_1\bar{u}_2)u_3\|_{L^2(I;L^2)}&\lesssim (N^*_1)^{1/2}\prod _{j=1}^3\| u_j\|_{L^6(I;L^6)}\lesssim (N^*_1)^{1/2}\prod _{j=1}^3\| \chi_Iu_j\|_{U^2_\Delta (\mathbb{R})}.
\end{align*}
This shows \eqref{est:tri1} with $C(\underline{N})=(N^*_1)^{1/2}$, which satisfies \eqref{bound:c}.

(Ib) [high$\times$high$\times$low$\to$high] $N^*_1\sim N^*_3\gg N^*_4$.

i) $N_1\sim N_2\gg N_3$.
In this case, we apply the standard bilinear Strichartz estimate (Lemma~\ref{lem:bs0}) to the product $u_1\bar{u}_2$, on which we may put $P_{\sim N^*_1}$.
Using the $L^\infty$ embedding 
\[ \| u_3\|_{L^\infty (I;L^\infty)}\lesssim N_3^{1/2}\| u_3\|_{L^\infty (I;L^2)}\lesssim N_3^{1/2}\| \chi_Iu_3\|_{U^2_\Delta (\mathbb{R})},\]
we have \eqref{est:tri1} with $C(\underline{N})=N_3^{1/2}$, which satisfies \eqref{bound:c}.

ii) $N_2\sim N_3\gg N_1$.
Noticing that $H(u_1\bar{u}_2)=u_1H(\bar{u}_2)$, we apply Lemma~\ref{lem:bs0} to the product $H(\bar{u}_2)u_3$ and follow the argument in the preceding case to obtain \eqref{est:tri1} with $C(\underline{N})=N_1^{1/2}$, which again satisfies \eqref{bound:c}.

iii) $N_1\sim N_3\gg N_2$.
In this case, we need to consider the dual estimate \eqref{est:tri2}, because we cannot use Lemma~\ref{lem:bs0} to the product $H(u_1)u_3$ (in the form of Remark~\ref{rem:bs}) when the Fourier supports of $u_1$ and $u_3$ are overlapping.
We first replace $u_4\in V^2_{rc,\Delta}(I_4)$ with its extension $\tilde{u}_4\in V^2_{-,rc,\Delta}(\mathbb{R})$ defined by $\tilde{u}_4(a):=\lim_{t\to a+0}u_4(t)$ and $\tilde{u}_4(t):=0$ for $t\not\in [a,b)$ (recall that $I_4=(a,b)$).
Next, we decompose $I_4$ into sub-intervals of length $\leq (N^*_1)^{-1}$, the number of which is $O(1)$.
Then, for each integral on a sub-interval $I=[a',b')$ we apply Lemma~\ref{lem:bs0} (in the form obtained in Remark~\ref{rem:V2}) to the product $u_3\partial_x\overline{\tilde{u}_4}$ (on which we may put $P_{\sim N^*_1}$), bound the remaining functions $H(u_1)$, $\bar{u}_2$ in the $L^\infty (I;L^2)$ and the $L^2(I;L^\infty)$ norms respectively, and finally derive the factor $(N^*_1)^{-1/2}N_2^{1/2}$ from the last one by the H\"older inequality in $t$ and the Bernstein inequality in $x$.
The resulting bound is
\begin{align*}
&\Big| \int_{I} \int H(u_1\overline{u_2})u_3 \partial_x\overline{\tilde{u}_4}\,dx\,dt\Big| \\
&\quad \lesssim (N^*_1)^{-(1/2)+}N_4(N^*_1)^{-1/2}N_2^{1/2}\| \chi_I\tilde{u}_4\|_{V^2_\Delta(\mathbb{R})}\prod _{j=1}^3\| \chi_Iu_j\|_{U^2_\Delta(\mathbb{R})}.
\end{align*}
(Since we have to bound $\tilde{u}_4$ in $V^2_{\Delta}$, the bilinear Strichartz estimate is accompanied by a factor $(N^*_1)^{0+}$.)
Now, it is verified directly from the definition of the $V^2_\Delta$ norm that
\[ \| \chi_I\tilde{u}_4\|_{V^2_\Delta(\mathbb{R})}\leq \| \tilde{u}_4\|_{V^2_\Delta(\mathbb{R})}\leq \sqrt{2}\| u_4\|_{V^2_\Delta(I_4)}.\]
As a result, we obtain \eqref{est:tri2} with $C(\underline{N})=(N^*_1)^{0+}N_2^{1/2}$, which satisfies \eqref{bound:c}.

(Ic) [high$\times$low$\times$low$\to$high] $N^*_1\sim N^*_2\gg N^*_3$.\\
We show \eqref{est:tri1} with $C(\underline{N})\lesssim (N^*_4)^{1/2}$.
If $N_1$ or $N_2\sim N^*_1$ (so that $N_1\not\sim N_2$), we can put $H$ on a single function.
Then, similarly to the case (Ib-i), we apply Lemma~\ref{lem:bs0} to the product of functions corresponding to $N^*_1$ and $N^*_3$ and use the $L^\infty$ embedding for the other one corresponding to $N^*_4$, to obtain the desired bound.
In the remaining case, i.e., if $N_3\sim N_4\gg N_1,N_2$, we apply the first modified bilinear Strichartz estimate (Lemma~\ref{lem:bs1}) to the left-hand side of \eqref{est:tri1}, which gives the same bound.

\medskip
\noindent
\textbf{Case (II)} $N_4\ll N^*_1$.

(IIa) [high$\times$high$\times$high$\to$low] $N_1\sim N_2\sim N_3\gg N_4$.\\
We follow the argument in the case (Ia) to obtain \eqref{est:tri1} with $C(\underline{N})=(N^*_1)^{1/2}$, which satisfies \eqref{bound:c}.

(IIb) [high$\times$high$\times$low$\to$low] $N^*_1\gg N^*_3$.

i) If $N_4\lesssim \min \{ N_1,N_2,N_3\}$, we show \eqref{est:tri1} with $C(\underline{N})\lesssim \min \{ N_1,N_2,N_3\}^{1/2}$ by considering the following two cases separately.
\begin{itemize}
\item If  $N_3\sim N^*_1$ (which implies $N_1\not\sim N_2$), we first bound the left-hand side of \eqref{est:tri1} by 
\[ (N^*_1)^{1/2}N_4^{1/2}\| H(u_1\bar{u}_2)u_3\|_{L^2(I;L^1)}\]
and then apply Lemma~\ref{lem:bs0} to $u_1\bar{u}_2$ (on which we may put $P_{\sim N^*_1}$).
This implies \eqref{est:tri1} with $C(\underline{N})=N_4^{1/2}$.
\item If $N_3\ll N^*_1$ (which implies $N_1\sim N_2\sim N^*_1$), we may put $P_{\lesssim N_3}$ on $u_1\bar{u}_2$.
Using the second modified bilinear Strichartz estimate (Lemma~\ref{lem:bs2}), we obtain \eqref{est:tri1} with $C(\underline{N})=N_3^{1/2}$.
\end{itemize}

ii) If $N_4\gg \min \{ N_1,N_2,N_3\}$, we consider the dual estimate \eqref{est:tri2}.
Note that we can always put $H$ on a single function, since $N_3\not\sim N_4$ and 
\[ \int_{I_4} \int H(u_1\overline{u_2})u_3 \partial_x\overline{u_4}\,dx\,dt=-\int_{I_4} \int u_1\overline{u_2}H(u_3 \partial_x\overline{u_4})\,dx\,dt.\]
Then, the argument is parallel to the case (Ib-iii).
This time we decompose $I_4$ into sub-intervals of length $\leq (N^*_1)^{-1}$, the number of which is $O(N^*_1/N_4)$, and apply Lemma~\ref{lem:bs0} to the product of functions corresponding to $N^*_1$ and $N^*_3\,(=N_4)$.
Further, we bound the remaining functions corresponding to $N^*_2$ and $N^*_4$ in the $L^\infty (I;L^2)$ and the $L^2(I;L^\infty)$ norms, respectively.
We then obtain \eqref{est:tri2} with
\[ C(\underline{N})\lesssim \frac{N^*_1}{N_4}\cdot (N^*_1)^{-(1/2)+}N_4(N^*_1)^{-1/2}(N^*_4)^{1/2}\lesssim (N^*_1)^{0+}(N^*_4)^{1/2},\]
which satisfies \eqref{bound:c}.

We have seen all the possible cases, and this completes the proof of the localized estimate \eqref{est:tri-local} with \eqref{bound:c}.
\end{proof}


\section{Energy estimate}\label{sec:energy}

In this section, we shall prove the following {\it a priori} estimate.
\begin{proposition}\label{prop:energy}
Assume $0<T\leq 1$ and that $u\in C([0,T];H^\infty)$ is a solution to \eqref{kdnls} on the time interval $[0,T]$.
Then, for $s>1/4$ there exist $\delta >0$ and $C>0$ (independent of $u$) such that if $\| u(0)\|_{L^2}\leq \delta$ then we have
\[ \| u\|_{E^s(T)}^2\leq C\Big( \| u(0)\|_{H^s}^2+\| u\|_{F^s(T)}^6\Big) .\]
\end{proposition}
In fact, this is the main part of the proof of Theorem~\ref{thm:main}. 
Recall that the $E^s(T)$ norm takes $L^\infty_t$ before the $\ell^2$ summation over dyadic frequency blocks, and so it is fairly stronger than the $L^\infty_tH^s_x$ norm.

\subsection{A reduction}

First of all, we reduce Proposition~\ref{prop:energy} to the following estimate on a ``modified energy''.
\begin{proposition}\label{prop:energy2}
Let $0<T\leq 1$ and $u\in C([0,T];H^\infty)$ be a solution to \eqref{kdnls} on the time interval $[0,T]$.
Let $s>1/4$, and assume that a smooth symbol $a\in C^\infty (\mathbb{R})$ has the following properties:
\begin{equation}\label{property:a}
\left\{ \begin{aligned}
&\text{$a$ is positive, even, non-decreasing in $[0,\infty )$, constant on $[-1,1]$,}\\
&a(2\xi )\lesssim a(\xi) \quad \text{for any $\xi >0$},\\
&\frac{a(N_1)}{a(N_2)}\gtrsim \Big( \frac{N_1}{N_2}\Big) ^{1/2} \quad \text{for any $N_1>N_2\geq 1$},\\
&|\partial_\xi ^ja(\xi )|\lesssim \langle \xi \rangle ^{-j}a(\xi )\quad \text{for any $\xi\in \mathbb{R}$ and $1\leq j\leq 5$}.
\end{aligned}\right.
\end{equation}
Then, there exist $\delta >0$ and $C>0$ depending on $s$ and the implicit constants in \eqref{property:a} (but not on $u$) such that if $\| u(0)\|_{L^2}\leq \delta$ then we have
\begin{align*}
E^a_0(u(t)):={}&\| \sqrt{a(D)}u(t)\|_{L^2}^2\\
\leq {}&C\Big( E^a_0(u(0))+\| u\|_{F^s(T)}^4\sum_{N\geq 1}a(N)\| P_Nu\|_{F_N(T)}^2\Big) ,\qquad t\in [0,T].
\end{align*}
\end{proposition}

\begin{remark}
(i) The Sobolev weight $a(\xi )=\langle \xi \rangle ^{2s}$ ($s\geq 1/4$) satisfies the conditions \eqref{property:a} (after modifying on $[-1,1]$).
With this choice of $a$, we can obtain from Proposition~\ref{prop:energy2} an $L^\infty([0,T];H^s)$ {\it a priori} estimate.
This is, however, weaker than what we want to prove in Proposition~\ref{prop:energy}. 

(ii) To obtain an $E^s(T)$ bound, one may consider estimating localized $H^s$ norms $N^{2s}\| \psi_N(D) u(t)\|_{L^2}^2$ for dyadic numbers $N\geq 1$ and summing them up. 
This is indeed the approach taken in \cite{IKT08}.
On the other hand, we will improve the bound by adding a correction term to the energy functional.
For this purpose, it will be convenient to introduce a modified energy $\| \sqrt{a(D)} u(t)\|_{L^2}^2$ and estimate it instead of the localized $H^s$ norms, where a symbol $a(\xi)$ is chosen so that it is positive everywhere and its derivatives are controlled by itself as $|\partial_\xi^ja(\xi)|\lesssim \langle \xi \rangle^{-j}a(\xi)$.
Such a modified energy has been used for the cubic NLS in \cite{KT07} and for the modified Benjamin-Ono (and also the DNLS) equation in \cite{G11,S21}.

(iii) Our choice of $a(\xi )$ is slightly simpler than that in \cite{KT07,G11,S21} (see the proof of Proposition~\ref{prop:energy} below).
Indeed, the modified energies in these papers are defined from a sequence of positive numbers $\{ \beta _N\}$ depending on the initial data, but we do not use such a sequence.
\end{remark}

\begin{proof}[Proof of Proposition~\ref{prop:energy} from Proposition~\ref{prop:energy2}]
Let $s>1/4$, $\varepsilon :=s-1/4>0$.
For each dyadic integer $\overline{N}$, we define the positive sequence $\{ a^{\overline{N}}_N\}_{N\geq 1}$ by
\[ a^{\overline{N}}_N:=N^{2s}\Big( \frac{N}{\overline{N}}\wedge \frac{\overline{N}}{N}\Big) ^{\varepsilon}=\left\{ \begin{alignedat}{2} &\overline{N}^{-\varepsilon}N^{2s+\varepsilon} &\quad &(N\leq \overline{N}), \\ &\overline{N}^{\varepsilon}N^{2s-\varepsilon} & &(N\geq \overline{N}).\end{alignedat}\right. \]
It is clear that $\{ a^{\overline{N}}_N\}$ is increasing in $N$.
In fact, the growth of $\{ a^{\overline{N}}_N\}$ is controlled as
\[ 2^{2s-\varepsilon}\leq \frac{a^{\overline{N}}_{2N}}{a^{\overline{N}}_N}\leq 2^{2s+\varepsilon}\qquad (N\geq 1).\]
Now, define the smooth symbol $a^{\overline{N}}\in C^\infty (\mathbb{R})$ by
\[ a^{\overline{N}}(\xi ):=\sum _{N\geq 1}a^{\overline{N}}_N\psi_N(\xi).\]
It is easy to see that $a^{\overline{N}}$ satisfies all the properties in \eqref{property:a} with implicit constants independent of $\overline{N}$.
Applying Proposition~\ref{prop:energy2} and restricting the left-hand side of the resulting estimate to the target frequency $\{ \langle \xi \rangle \sim \overline{N}\}$, we have
\[ \sup _{t\in [0,T]}\overline{N}^{2s}\| P_{\overline{N}}u(t)\|_{L^2}^2\lesssim \sum _{N\geq 1}\Big( \frac{N}{\overline{N}}\wedge \frac{\overline{N}}{N}\Big) ^{\varepsilon}N^{2s}\Big[ \| P_Nu(0)\|_{L^2}^2+\| u\|_{F^s(T)}^4\| P_Nu\|_{F_N(T)}^2\Big] \]
for any smooth solution $u\in C([0,T];H^\infty )$ to \eqref{kdnls} with $\| u(0)\|_{L^2}$ sufficiently small.
Summing up in $\overline{N}$, we obtain the claimed estimate.
\end{proof}

\subsection{Construction of the modified energies}

Now, we start proving Proposition~\ref{prop:energy2}.
The argument is very similar to the estimate of the modified energy with correction terms in the $I$-method, where an important role is played by various cancellations after symmetrization of the energy functionals (see, e.g., Colliander, Keel, Staffilani, Takaoka and Tao \cite{CKSTT}).
However, there are less symmetries compared to the DNLS case, and more delicate analysis is required.
In particular, some of the highest order terms cannot be canceled out, and we need to recognize these terms to be non-positive by making use of the dissipative nature of the equation. 

Let $a\in C^\infty (\mathbb{R})$ be a symbol satisfying \eqref{property:a}.
Hereafter, the notation $\xi_{ij\ldots}=\xi_i+\xi_j+\cdots$ will be frequently used.
Our proof will be designed for the periodic problem; however, in view of Remark~\ref{rem:nonperiodic}, the same argument can be applied in the non-periodic setting.
For a smooth solution $u$ of \eqref{kdnls}, we have
\begin{align*}
\partial_t\hat{u}(\xi) &= -i\xi ^2\hat{u}(\xi) +i\alpha \xi \sum _{\xi =\xi _{123}}\hat{u}(\xi_1)\hat{\bar{u}}(\xi _2)\hat{u}(\xi_3) +\beta \xi \sum _{\xi =\xi_{123}}\mathrm{sgn}(\xi_{12})\hat{u}(\xi_1)\hat{\bar{u}}(\xi _2)\hat{u}(\xi_3),\\
\partial_t\hat{\bar{u}}(\xi) &= i\xi ^2\hat{\bar{u}}(\xi) +i\alpha \xi \sum _{\xi =\xi _{123}}\hat{\bar{u}}(\xi_1)\hat{u}(\xi _2)\hat{\bar{u}}(\xi_3) +\beta \xi \sum _{\xi =\xi_{123}}\mathrm{sgn}(\xi_{23})\hat{\bar{u}}(\xi_1)\hat{u}(\xi _2)\hat{\bar{u}}(\xi_3).
\end{align*}
The derivative of $E^a_0(u(t))=\sum _{\xi _{12}=0}a(\xi_1)\hat{u}(\xi_1)\hat{\bar{u}}(\xi_2)$ is computed as
\begin{align*}
\frac{d}{dt}E^a_0(u(t))
&=i\alpha \sum _{\xi_{1234}=0}\Big[ a(\xi_{123})\xi_{123} +a(\xi_1)\xi_{234}\Big] \hat{u}(\xi_1)\hat{\bar{u}}(\xi_2)\hat{u}(\xi_3)\hat{\bar{u}}(\xi_4)\\
&\quad +\beta \sum _{\xi_{1234}=0}\Big[ a(\xi_{123})\xi_{123}\mathrm{sgn}(\xi_{12}) +a(\xi_1)\xi_{234}\mathrm{sgn}(\xi_{34})\Big] \hat{u}(\xi_1)\hat{\bar{u}}(\xi_2)\hat{u}(\xi_3)\hat{\bar{u}}(\xi_4).
\end{align*}
The first term is the same as that appears in the DNLS case, and it is symmetrized as follows:
\[ \frac{\alpha}{2i}\sum _{\xi_{1234}=0}\Big[ \xi_1a(\xi_1)+\xi_2a(\xi_2)+\xi_3a(\xi_3)+\xi_4a(\xi_4)\Big] \hat{u}(\xi_1)\hat{\bar{u}}(\xi_2)\hat{u}(\xi_3)\hat{\bar{u}}(\xi_4) .\]
We observe that the multiplier part $\xi_1a(\xi_1)+\dots +\xi_4a(\xi_4)$ vanishes for the resonant frequencies:
\[ \xi_{1234}=0,\qquad \xi_1^2-\xi_2^2+\xi_3^2-\xi_4^2=-2\xi_{12}\xi_{23}=0.\]
Then, this quadrilinear term can be canceled with the leading term of the derivative of the quadrilinear functional
\[ -\frac{\alpha}{2i} \sum _{\begin{smallmatrix}\xi_{1234}=0\\ \xi_{12}\xi_{23}\neq 0\end{smallmatrix}}\frac{\xi_1a(\xi_1)+\xi_2a(\xi_2)+\xi_3a(\xi_3)+\xi_4a(\xi_4)}{-i(\xi_1^2-\xi_2^2+\xi_3^2-\xi_4^2)}\hat{u}(\xi_1)\hat{\bar{u}}(\xi_2)\hat{u}(\xi_3)\hat{\bar{u}}(\xi_4) ,\]
which can be used as an appropriate correction term to $E^a_0(u)$.

In the following, we assume $\alpha=0$ for simplicity and consider the term
\[ \beta \sum _{\xi_{1234}=0}\Big[ a(\xi_{123})\xi_{123}\mathrm{sgn}(\xi_{12}) +a(\xi_1)\xi_{234}\mathrm{sgn}(\xi_{34})\Big] \hat{u}(\xi_1)\hat{\bar{u}}(\xi_2)\hat{u}(\xi_3)\hat{\bar{u}}(\xi_4). \]
This term has less symmetry due to the sign functions.
In fact, this is symmetrized as
\[ \frac{\beta}{2}\sum _{\xi_{1234}=0}\Big[ \big( \xi_1a(\xi_1)+\xi_2a(\xi_2)\big) \mathrm{sgn}(\xi_{12}) +\big( \xi_3a(\xi_3)+\xi_4a(\xi_4)\big) \mathrm{sgn}(\xi_{34}) \Big] \hat{u}(\xi_1)\hat{\bar{u}}(\xi_2)\hat{u}(\xi_3)\hat{\bar{u}}(\xi_4), \]
and the multiplier part does not vanish when $\xi_{23}=0\neq \xi_{12}$ (in this case $\mathrm{sgn}(\xi_{34}) =-\mathrm{sgn}(\xi_{12}) \neq 0$).
Now, we observe that the function $\xi \mapsto \xi a(\xi)$ is odd and strictly increasing on $\mathbb{R}$, and hence 
\[ \big( \xi_1a(\xi_1)+\xi_2a(\xi_2)\big) \mathrm{sgn}(\xi_{12}) = \big| \xi_1a(\xi_1)+\xi_2a(\xi_2)\big| >0\]
for any $\xi_1,\xi_2\in \mathbb{R}$ with $\xi_{12}\neq 0$.
Then, we decompose this term as
\begin{align*}
&\beta \sum _{\xi_{1234}=0}\sqrt{|\xi_1a(\xi_1)+\xi_2a(\xi_2)|} \sqrt{|\xi_3a(\xi_3)+\xi_4a(\xi_4)|} \hat{u}(\xi_1)\hat{\bar{u}}(\xi_2)\hat{u}(\xi_3)\hat{\bar{u}}(\xi_4)\\
&+\frac{\beta}{2}\sum _{\xi_{1234}=0}\Big[ \sqrt{|\xi_1a(\xi_1)+\xi_2a(\xi_2)|} -\sqrt{|\xi_3a(\xi_3)+\xi_4a(\xi_4)|} \Big] ^2\hat{u}(\xi_1)\hat{\bar{u}}(\xi_2)\hat{u}(\xi_3)\hat{\bar{u}}(\xi_4)\\
&=:\mathcal{Q}_1+\mathcal{Q}_2.
\end{align*}
On one hand, for $\beta <0$ we have
\[ \mathcal{Q}_1=\beta \sum _{\xi}\Big| \sum _{\xi_{12}=\xi}\sqrt{|\xi_1a(\xi_1)+\xi_2a(\xi_2)|}\hat{u}(\xi_1)\hat{\bar{u}}(\xi_2)\Big|^2 \leq 0.\]
On the other hand, it will turn out that the multiplier part of $\mathcal{Q}_2$ vanishes when $\xi_{12}=0$ and also when $\xi_{23}=0$.
$\mathcal{Q}_2$ is then canceled out by adding the correction term
\[ E^a_1(u):=\sum _{\begin{smallmatrix}\xi_{1234}=0\\ \xi_{12}\xi_{23}\neq 0\end{smallmatrix}}b^a_4(\xi_1,\xi_2,\xi_3,\xi_4)\mathrm{sgn}(\xi_{12})\hat{u}(\xi_1)\hat{\bar{u}}(\xi_2)\hat{u}(\xi_3)\hat{\bar{u}}(\xi_4) \]
to the modified energy $E^a_0(u)$, where
\[ b^a_4(\xi_1,\xi_2,\xi_3,\xi_4):=-\frac{\beta}{2}\frac{\big[ \sqrt{|\xi_1a(\xi_1)+\xi_2a(\xi_2)|} -\sqrt{|\xi_3a(\xi_3)+\xi_4a(\xi_4)|} \big] ^2}{2i\xi_{12}\xi_{23}\mathrm{sgn}(\xi_{12})}.\]
We can show that $b^a_4$ is extended to a smooth function on $\Gamma_4$, where
\[ \Gamma_m=\{ (\xi_j)_{1\leq j\leq m}:\xi_{12\dots m}=0\} \]
(we put $\mathrm{sgn}(\xi_{12})$ outside in order to make $b^a_4$ smooth).
Moreover, $\xi_{12}\xi_{23}=0$ implies $b^a_4(\xi_1,\xi_2,\xi_3,\xi_4)=0$ (and hence the restriction $\xi_{12}\xi_{23}\neq 0$ for the sum in $E^a_1(u)$ can be disregarded).
In fact, noticing that
\[ q(\xi_1,\xi_2):=\frac{\xi_1a(\xi_1)+\xi_2a(\xi_2)}{\xi_1+\xi_2}>0\qquad (\xi_1,\xi_2\in \mathbb{R},~\xi_{12}\neq 0),\]
on $\Gamma_4\cap \{ \xi_{12}\xi_{23}\neq 0\}$ we compute it as
\begin{align*}
b^a_4(\xi_1,\xi_2,\xi_3,\xi_4)
&=\frac{i\beta}{4\xi_{23}}\bigg[ \sqrt{\frac{|\xi_1a(\xi_1)+\xi_2a(\xi_2)|}{|\xi_{12}|}}-\sqrt{\frac{|\xi_3a(\xi_3)+\xi_4a(\xi_4)|}{|\xi_{34}|}}\bigg] ^2\\
&=\frac{i\beta}{4\xi_{23}}\big[ \sqrt{q(\xi_1,\xi_2)}-\sqrt{q(\xi_3,\xi_4)}\big] ^2\\ 
&=\frac{i\beta}{4\xi_{23}}\Big[ \frac{q(\xi_1,\xi_2)-q(\xi_1+\xi_{23},\xi_2-\xi_{23})}{\sqrt{q(\xi_1,\xi_2)}+\sqrt{q(\xi_3,\xi_4)}}\Big] ^2\\
&=\frac{i\beta\xi_{23}}{4}\frac{\Big[ \displaystyle\int _0^1\big( \partial_1q-\partial_2q\big) (\xi_1+\xi_{23}t,\xi_2-\xi_{23}t)\,dt\Big] ^2}{\big[ \sqrt{q(\xi_1,\xi_2)}+\sqrt{q(\xi_3,\xi_4)}\big] ^2},
\end{align*}
while $q(\xi_1,\xi_2)$ is actually positive and smooth on $\mathbb{R}^2$, as we see in the following lemma.
\begin{lemma}\label{lem:q}
$q(\xi_1,\xi_2)$ can be extended to a smooth positive function on $\mathbb{R}^2$.
Moreover, the following holds (with $\xi_{\max}:=|\xi_1|\vee |\xi_2|$):

\noindent
$\mathrm{(i)}$ $q(\xi_1,\xi_2)\!\sim \!a(\xi_{\max})$, $\big| \big[ \partial_1^{\gamma_1}\partial_2^{\gamma_2}q\big] (\xi_1,\xi_2)\big| \!\lesssim\! \langle \xi_1\rangle^{-\gamma_1}\langle \xi_2\rangle^{-\gamma_2}a(\xi_{\max})$ $(1\!\leq\! |\gamma|\!\leq\! 3)$.

\noindent
$\mathrm{(ii)}$ $\big| \big[ \partial_1^{\gamma_1}\partial_2^{\gamma_2}(\partial_1-\partial_2)q\big](\xi_1,\xi_2)\big| \lesssim \langle \xi_1\rangle^{-\gamma_1}\langle \xi_2\rangle^{-\gamma_2}\langle \xi_{\max}\rangle^{-1}a(\xi_{\max})$ $(0\leq |\gamma|\leq 3)$.
\end{lemma}

\begin{proof}
For $(\xi_1,\xi_2)\in \mathbb{R}^2\setminus \{ \xi_{12}=0\}$, we have
\begin{align}
q(\xi_1,\xi_2)&=\frac{\xi_1a(\xi_1)+\xi_2a(\xi_2)}{\xi_1+\xi_2} \label{def:q1} \\
&=\frac{\xi_1a(\xi_1)-(-\xi_2)a(-\xi_2)}{\xi_1-(-\xi_2)}=\int _0^1\big( \xi a(\xi )\big)' (-\xi_2+\xi_{12}t)\,dt. \label{def:q2}
\end{align}
Since $a$ is smooth and $(\xi a(\xi ))'=a(\xi)+\xi a'(\xi) \geq a(\xi )\geq a(0)>0$ ($\xi \in \mathbb{R}$), \eqref{def:q2} defines a positive smooth function on $\mathbb{R}^2$.

To show $q(\xi_1,\xi_2)\sim a(\xi_{\max})$, it suffices to consider the following three cases.
If $\xi_{\max}\leq 1$, then $q(\xi_1,\xi_2)=a(0)=a(\xi_{\max})$.
If $\xi_{\max}>1$ and $|\xi_{12}|\sim \xi_{\max}$, then the claim follows from the expression \eqref{def:q1} (and some more argument).
If $\xi_{\max}>1$ and $|\xi_{12}|\ll \xi_{\max}$, then we have $|-\xi_2+\xi_{12}t|\sim |\xi_2|\sim \xi_{\max}$ for $t\in [0,1]$, and the claim follows from \eqref{def:q2} since $(\xi a(\xi))'=a(\xi )+\xi a'(\xi) \sim a(\xi)$.

For the derivatives of $q$, we may focus on the case $\xi_{\max}>1$.
In the case $|\xi_{12}|\sim \xi_{\max}$, we first observe that
\[ \big| \partial_1^{\gamma_1}\partial_2^{\gamma_2}\big[ \xi_{12}^{-1}\big] \big| \lesssim \langle \xi_{\max}\rangle ^{-1-|\gamma|}.\]
On the other hand, using the property $|\partial_\xi ^ja(\xi )|\lesssim \langle \xi \rangle ^{-j}a(\xi )$ for $j\leq 5$, we have 
\[ |\partial_\xi^j(\xi a(\xi ))|\lesssim \langle \xi \rangle ^{1-j}a(\xi),\]
which implies that
\[ \big| \partial_1^{\gamma_1}\partial_2^{\gamma_2}(\xi_1a(\xi_1)+\xi_2a(\xi_2))\big| \left\{ \begin{alignedat}{2}
&\lesssim \langle \xi _1\rangle ^{1-j}a(\xi_{\max}) &\quad &(\gamma=(j,0)),\\
&\lesssim \langle \xi _2\rangle ^{1-j}a(\xi_{\max}) &\quad &(\gamma=(0,j)),\\
&=0 & &(\text{$\gamma_1\geq 1$ and $\gamma_2\geq 1$}).
\end{alignedat}\right.\]
The claimed estimate follows from these estimates and the expression \eqref{def:q1}.
When $|\xi_{12}|\ll \xi_{\max}$, we deduce from the expression \eqref{def:q2} that
\[ \big[ \partial_1^{\gamma_1}\partial_2^{\gamma_2}q\big] (\xi_1,\xi_2) =\int _0^1 t^{\gamma_1}(t-1)^{\gamma_2}\big[ \partial_\xi^{1+|\gamma|}(\xi a(\xi))\big] (-\xi_2+\xi_{12}t)\,dt.\]
Recalling that $|-\xi_2+\xi_{12}t|\sim \xi_{\max}$ for $t\in [0,1]$, we have
\[ \big| \big[ \partial_1^{\gamma_1}\partial_2^{\gamma_2}q\big] (\xi_1,\xi_2)\big| \lesssim \langle \xi_{\max} \rangle ^{-|\gamma|}a(\xi_{\max})\lesssim \langle \xi_1 \rangle ^{-\gamma_1}\langle \xi_2\rangle ^{-\gamma_2}a(\xi_{\max}).\]
This proves (i).
For (ii), we compute
\begin{align*}
(\partial_1q-\partial_2q)(\xi_1,\xi_2)&=\frac{a(\xi_1)-a(\xi_2)}{\xi_{12}}+\frac{\xi_1a'(\xi_1)-\xi_2a'(\xi_2)}{\xi_{12}}\\
&=\int _0^1\big( \xi a(\xi )\big)'' (-\xi_2+\xi_{12}t)\,dt.
\end{align*}
Using these expressions instead of \eqref{def:q1}--\eqref{def:q2}, the desired estimate is verified by a similar argument to (i).
\end{proof}

We have the following estimate on $E^a_1(u(t))$ for each $t$.

\begin{lemma}\label{lem:pointwise}
We have
\[ |E^a_1(f)|\lesssim \| f\|_{L^2}^2E^a_0(f)\]
for any $f\in L^2_x$ such that $E^a_0(f)<\infty$.
\end{lemma}

\begin{proof}
Let us begin with the dyadic decomposition:
\begin{align*}
|E^a_1(f)|\lesssim \sum_{N_1,\dots ,N_4\geq 1}\sum _{\Gamma_4}|b^a_4(\xi_1,\xi_2,\xi_3,\xi_4)|\cdot \psi_{N_1}(\xi_1)|\hat{f}(\xi_1)|\cdots \psi_{N_4}(\xi_4)|\hat{\bar{f}}(\xi_4)|.
\end{align*}
We can show that
\[ |b^a_4(\xi_1,\xi_2,\xi_3,\xi_4)|\lesssim \frac{a(N^*_1)}{N^*_1}\qquad \text{on}~(I_{N_1}\times\cdots \times I_{N_4})\cap \Gamma_4,\]
where we renumber $N_1,\dots ,N_4$ as $N^*_1,\dots ,N^*_4$ such that $N_1^*\geq N_2^*\geq N_3^*\geq N_4^*$. 
(We will actually prove a stronger result including estimates on derivatives of $b^a_4$ in the proof of Lemma~\ref{lem:separation} below.)
Then, by H\"older we have
\begin{align*}
|E^a_1(f)|&\lesssim \sum_{N^*_1\sim N^*_2\geq N^*_3\geq N^*_4\geq 1}\frac{a(N^*_1)}{N^*_1}(N^*_3N^*_4)^{1/2}\prod _{j=1}^4\| P_{N^*_j}f\|_{L^2}\\
&\lesssim \| f\|_{L^2}^2\sum_{N^*_1\sim N^*_2}a(N^*_1)\| P_{N^*_1}f\|_{L^2}\| P_{N^*_2}f\|_{L^2}\\
&\lesssim \| f\|_{L^2}^2\sum _{N}a(N)\| P_Nf\|_{L^2}^2\\
&\lesssim \| f\|_{L^2}^2E^a_0(f),
\end{align*}
as desired.
\end{proof}

By differentiating $E^a_1(u(t))$ in $t$ and substituting the equation, we obtain
\begin{align*}
\frac{d}{dt}\Big( E^a_0(u(t))+E^a_1(u(t))\Big) &=\mathcal{Q}_1+\beta \sum_{\Gamma_6}\hat{u}(\xi_1)\hat{\bar{u}}(\xi_2)\hat{u}(\xi_3)\hat{\bar{u}}(\xi_4)\hat{u}(\xi_5)\hat{\bar{u}}(\xi_6) \\
&\qquad\qquad \times \Big[ -b^a_4(\xi_{123},\xi_4,\xi_5,\xi_6)\mathrm{sgn}(\xi_{56})\xi_{123}\mathrm{sgn}(\xi_{12}) \\
&\qquad\qquad\quad -b^a_4(\xi_1,\xi_{234},\xi_5,\xi_6)\mathrm{sgn}(\xi_{56})\xi_{234}\mathrm{sgn}(\xi_{34})\\
&\qquad\qquad\quad +b^a_4(\xi_1,\xi_2,\xi_{345},\xi_6)\mathrm{sgn}(\xi_{12})\xi_{345}\mathrm{sgn}(\xi_{34})\\
&\qquad\qquad\quad +b^a_4(\xi_1,\xi_2,\xi_3,\xi_{456})\mathrm{sgn}(\xi_{12})\xi_{456}\mathrm{sgn}(\xi_{56}) \Big] \\
&=: \mathcal{Q}_1+\beta \big( \mathcal{R}_1+\dots +\mathcal{R}_4\big) .
\end{align*}
Here, it turns out that $\mathcal{R}_1=\mathcal{R}_3=\overline{\mathcal{R}_2}=\overline{\mathcal{R}_4}$.
To see this, we start with $\mathcal{R}_1$ and first change variables as $(\xi_1,\xi_2,\xi_3,\xi_4,\xi_5,\xi_6)\mapsto (\xi_3,\xi_4,\xi_5,\xi_6,\xi_1,\xi_2)$ to obtain $\mathcal{R}_3$.
We then see $\overline{\mathcal{R}_1}=\mathcal{R}_4$ and $\overline{\mathcal{R}_2}=\mathcal{R}_3$ by taking the complex conjugate, using $\overline{\hat{u}(\xi)}=\hat{\bar{u}}(-\xi )$, and changing variables as $(\xi_1,\xi_2,\xi_3,\xi_4,\xi_5,\xi_6)\mapsto (-\xi_6,-\xi_5,-\xi_4,-\xi_3,-\xi_2,-\xi_1)$.
Therefore, it suffices to consider
\begin{align*}
&\mathcal{R}(u):=-\mathcal{R}_1\\
&=\sum_{\Gamma_6}b^a_4(\xi_{123},\xi_4,\xi_5,\xi_6)\xi_{123}\mathrm{sgn}(\xi_{12})\mathrm{sgn}(\xi_{56})\hat{u}(\xi_1)\hat{\bar{u}}(\xi_2)\hat{u}(\xi_3)\hat{\bar{u}}(\xi_4)\hat{u}(\xi_5)\hat{\bar{u}}(\xi_6),
\end{align*}
which satisfies
\[ \frac{d}{dt}\Big( E^a_0(u(t))+E^a_1(u(t))\Big) =\mathcal{Q}_1(u(t))-4\beta \Re \mathcal{R}(u(t)).\]

We need to prove the following estimate on the remainder term $\mathcal{R}(u)$, which is the hardest part in the overall proof of the main theorem.
\begin{proposition}\label{prop:R6}
Let $s>1/4$.
For $0<T\leq 1$, we have
\[ \Big| \int _0^T \mathcal{R}(u(t))\,dt\Big| \lesssim \| u\|_{F^s(T)}^4\sum _{N\geq 1}a(N)\| P_Nu\|_{F_N(T)}^2.\]
\end{proposition}
Let us postpone the proof of this proposition and verify Proposition~\ref{prop:energy2}.
\begin{proof}[Proof of Proposition~\ref{prop:energy2}]
By Lemma~\ref{lem:pointwise}, Proposition~\ref{prop:R6}, and the fact that any smooth solution of \eqref{kdnls} reduces its $L^2$ norm, we have
\begin{align*}
&\sup_{t\in [0,T]}E^a_0(u(t))\\
&\leq E^a_0(u_0)+\sup_{t\in [0,T]}|E^a_1(u(t))|+C\| u\|_{F^s(T)}^4\sum _{N\geq 1}a(N)\| P_Nu\|_{F_N(T)}^2\\
&\leq E^a_0(u_0)+C\| u(0)\|_{L^2}^2\sup_{t\in [0,T]}E^a_0(u(t))+C\| u\|_{F^s(T)}^4\sum _{N\geq 1}a(N)\| P_Nu\|_{F_N(T)}^2.
\end{align*}
Assuming $\| u(0)\|_{L^2}\ll 1$, this yields the claimed estimate.
\end{proof}

\subsection{Estimate on the remainder term}

It remains to prove Proposition~\ref{prop:R6}.
A difficulty here is that we cannot directly apply pointwise bounds on the multipliers (as we did in the proof of Lemma~\ref{lem:pointwise} above), because $u\in F^s(T)$ does not in general imply $\mathcal{F}^{-1}_\xi [|\hat{u}(t,\xi )|]\in F^s(T)$.
Indeed, linear solutions $u=e^{it\partial_x^2}u_0$ can be considered as counterexamples. 

We prepare the following lemma, which allows us to separate variables in the multiplier $b^a_4$.
This idea has also been used in \cite{KT07,H12,S21}. 
\begin{lemma}\label{lem:separation}
Let $N_1,\dots ,N_4$ be dyadic integers such that $N^*_1\sim N^*_2$, where $N^*_1,\dots ,N^*_4$ denote the numbers $N_1,\dots ,N_4$ rearranged in decreasing order.
Let $\tilde{\psi}_1,\dots ,\tilde{\psi}_4\in C^\infty_0(\mathbb{R})$ be bump functions such that $\mathrm{supp}\,(\tilde{\psi}_j(\cdot /N_j))\subset \mathcal{I}_{N_j}$, with $\mathcal{I}_{N_j}$ defined as in Definition~\ref{defn:spaces} (i.e., $\tilde{\psi}_j$ is supported in $[-2,2]$ if $N_j=1$ and in $[-2,2]\setminus (-\frac12,\frac12)$ if $N_j>1$).

Then, there is a sequence $\hat{c}\in \ell ^1(\mathbb{Z}^4)$ such that
\begin{align*}
&b^a_4(\xi_1,\xi_2,\xi_3,\xi_4)\xi_1\cdot \tilde{\psi}_1\Big( \frac{\xi_1}{N_1}\Big) \tilde{\psi}_2\Big( \frac{\xi_2}{N_2}\Big) \tilde{\psi}_3\Big( \frac{\xi_3}{N_3}\Big) \tilde{\psi}_4\Big( \frac{\xi_4}{N_4}\Big) \\
&=\sum _{k_1,\dots ,k_4\in \mathbb{Z}}\hat{c}(k_1,k_2,k_3,k_4)e^{i\big( \frac{k_1}{N_1}\xi_1+\frac{k_2}{N_2}\xi_2+\frac{k_3}{N_3}\xi_3+\frac{k_4}{N_4}\xi_4\big)},\qquad (\xi _1,\dots ,\xi_4)\in \Gamma_4
\end{align*}
and
\[ \sum _{k_1,\dots ,k_4\in \mathbb{Z}}\big| \hat{c}(k_1,k_2,k_3,k_4)\big| \lesssim N_1\frac{a(N^*_1)}{N^*_1}.\]
\end{lemma}

\begin{proof}
Following the argument in \cite{H12}, we first construct a smooth function $\tilde{b}^a_4(\xi_1,\dots ,\xi_4)$ on $\mathcal{I}_{N_1}\times \cdots \times \mathcal{I}_{N_4}$ which extends $b^a_4(\xi_1,\dots ,\xi_4)$ (defined on $\Gamma_4$) and satisfies
\begin{equation}\label{est:b}
\big| \partial_1^{\gamma_1}\partial_2^{\gamma_2}\partial_3^{\gamma_3}\partial_4^{\gamma_4}\tilde{b}^a_4(\xi_1,\xi_2,\xi_3,\xi_4)\big| \lesssim \frac{1}{N_1^{\gamma_1}N_2^{\gamma_2}N_3^{\gamma_3}N_4^{\gamma_4}}\frac{a(N^*_1)}{N^*_1}\quad (0\leq |\gamma |\leq 3).
\end{equation}
We use the following extensions of $b^a_4$:
\begin{align*}
\tilde{b}_1(\xi_1,\xi_2,\xi_3,\xi_4)&=\frac{i\beta}{4\xi_{23}}\frac{\big[ q(\xi_1,\xi_2)-q(\xi_3,\xi_4)\big] ^2}{\big[ \sqrt{q(\xi_1,\xi_2)}+\sqrt{q(\xi_3,\xi_4)}\big] ^2}\qquad (\text{on}~\mathbb{R}^4\setminus \{ \xi_{23}=0\} ),\\
\tilde{b}_2(\xi_1,\xi_2,\xi_3,\xi_4)&=\frac{i\beta \xi_{23}}{4\xi_{12}^2}\frac{\big[ q(\xi_2,\xi_3)-q(\xi_1,\xi_4)\big] ^2}{\big[ \sqrt{q(\xi_1,\xi_2)}+\sqrt{q(\xi_3,\xi_4)}\big] ^2}\qquad (\text{on}~\mathbb{R}^4\setminus \{ \xi_{12}=0\} ),\\
\tilde{b}_3(\xi_1,\xi_2,\xi_3,\xi_4)&=\frac{i\beta \xi_{23}}{4}\frac{\Big[ \displaystyle\int _0^1\big( \partial_1q-\partial_2q\big) (\xi_1+\xi_{23}t,\xi_2-\xi_{23}t)\,dt\Big] ^2}{\big[ \sqrt{q(\xi_1,\xi_2)}+\sqrt{q(\xi_3,\xi_4)}\big] ^2}\qquad (\text{on}~\mathbb{R}^4).
\end{align*}
From Lemma~\ref{lem:q}, we can show that
\begin{align*}
\bullet ~&\left| \partial_1^{\gamma_1}\partial_2^{\gamma_2}\partial_3^{\gamma_3}\partial_4^{\gamma_4}\begin{pmatrix} \big[ q(\xi_1,\xi_2)-q(\xi_3,\xi_4)\big] ^2\\ \big[ q(\xi_2,\xi_3)-q(\xi_1,\xi_4)\big] ^2\end{pmatrix} \right| \lesssim \frac{a(N^*_1)^2}{N_1^{\gamma_1}N_2^{\gamma_2}N_3^{\gamma_3}N_4^{\gamma_4}};\\
\bullet ~&\big| \partial_1^{\gamma_1}\partial_2^{\gamma_2}\sqrt{q(\xi_1,\xi_2)}\big| \lesssim \langle \xi_1\rangle^{-\gamma_1}\langle \xi_2\rangle^{-\gamma_2}\sqrt{a(|\xi_1|\vee |\xi_2|)}\qquad (\xi_1,\xi_2\in \mathbb{R});\\
\bullet ~&\big| \partial_1^{\gamma_1}\partial_2^{\gamma_2}\partial_3^{\gamma_3}\partial_4^{\gamma_4}\big[ \sqrt{q(\xi_1,\xi_2)}+\sqrt{q(\xi_3,\xi_4)}\big] ^{-2}\big| \lesssim \frac{1}{N_1^{\gamma_1}N_2^{\gamma_2}N_3^{\gamma_3}N_4^{\gamma_4}a(N^*_1)};\\
\bullet ~&\text{If $N_1\sim N_2\sim N^*_1$ and $|\xi_{23}|\ll N^*_1$, then}\\
&\left| \partial_1^{\gamma_1}\partial_2^{\gamma_2}\partial_3^{\gamma_3}\partial_4^{\gamma_4}\Big[ \displaystyle\int _0^1\big( \partial_1q-\partial_2q\big) (\xi_1+\xi_{23}t,\xi_2-\xi_{23}t)\,dt\Big] ^2\right| \lesssim \frac{a(N^*_1)^2}{(N^*_1)^{2+|\gamma|}}.
\end{align*}
Using these bounds, we see that the desired estimates \eqref{est:b} hold
\begin{itemize}
\item for $\tilde{b}_1$ if $|\xi_{23}|\sim N^*_1$;
\item for $\tilde{b}_2$ if $|\xi_{12}|\sim N^*_1$;
\item for $\tilde{b}_3$ if $N_1\sim N_2\sim N^*_1$ and $|\xi_{23}|\ll N^*_1$.
\end{itemize}
Note that $|\xi_{12}|\vee |\xi_{23}|\ll N^*_1$ implies $N_1\sim N_2\sim N_3\sim N^*_1$ under the hypothesis $N^*_1\sim N^*_2$.
Therefore, we can define $\tilde{b}^a_4$ by
\[ \tilde{b}^a_4:=\bigg[ 1- \eta \Big( \frac{\xi _{23}}{\tfrac{1}{100}N^*_1}\Big) \bigg] \tilde{b}_1+\bigg[ 1- \eta \Big( \frac{\xi _{12}}{\tfrac{1}{100}N^*_1}\Big) \bigg] \eta \Big( \frac{\xi _{23}}{\tfrac{1}{100}N^*_1}\Big) \tilde{b}_2+\eta \Big( \frac{\xi _{12}}{\tfrac{1}{100}N^*_1}\Big) \eta \Big( \frac{\xi _{23}}{\tfrac{1}{100}N^*_1}\Big) \tilde{b}_3,\]
for instance, where $\eta$ is defined as in Definition~\ref{defn:spaces}.
It is clear that the above defined $\tilde{b}^a_4$ coincides with $b^a_4$ on $\Gamma_4$ and satisfies \eqref{est:b}.

Now, we define
\[ c(\eta_1,\eta_2,\eta_3,\eta_4):=\tilde{b}^a_4(N_1\eta_1,N_2\eta_2,N_3\eta_3,N_4\eta_4)N_1\eta_1\cdot \tilde{\psi}_1(\eta_1)\tilde{\psi}_2(\eta_2)\tilde{\psi}_3(\eta_3)\tilde{\psi}_4(\eta_4),\]
which is a smooth function supported in $[-2,2]^4$ and thus can be extended to a $2\pi$-periodic smooth function on $\mathbb{R}^4$.
Let $\hat{c}(k_1,k_2,k_3,k_4)$ be the Fourier coefficients of $c$, then the claimed identity follows from the Fourier series expansion and the restriction onto $\Gamma_4$.
Moreover, we deduce from \eqref{est:b} that
\[ \| c\|_{C^3([-\pi ,\pi]^4)}:=\max_{\eta \in [-\pi ,\pi ]^4,\,|\gamma|\leq 3}\big| \partial_1^{\gamma_1}\partial_2^{\gamma_2}\partial_3^{\gamma_3}\partial_4^{\gamma_4}c(\eta_1,\eta_2,\eta_3,\eta_4)\big| \lesssim N_1\frac{a(N^*_1)}{N^*_1},\]
which then implies that
\[ \| \hat{c}\|_{\ell ^1(\mathbb{Z}^4)}\lesssim N_1\frac{a(N^*_1)}{N^*_1}.\]
This completes the proof.
\end{proof}

We are now in a position to prove Proposition~\ref{prop:R6}.
\begin{proof}[Proof of Proposition~\ref{prop:R6}]
As usual, we decompose the sum into dyadic pieces:
\begin{align*}
&\int _0^T \mathcal{R}(u(t))\,dt\\
&=\sum _{N_1,\dots ,N_6,N\geq 1}\int_0^T \sum_{\Gamma_6}b^a_4(\xi_{123},\xi_4,\xi_5,\xi_6)\xi_{123}\psi_{N}(\xi_{123})\tilde{\psi}_{N_4}(\xi_4)\tilde{\psi}_{N_5}(\xi_5)\tilde{\psi}_{N_6}(\xi_6)\\
&\qquad\qquad \times \mathrm{sgn}(\xi_{12})\mathrm{sgn}(\xi_{56})(\psi_{N_1}\hat{u})(t,\xi_1)(\psi_{N_2}\hat{\bar{u}})(t,\xi_2)\cdots (\psi_{N_6}\hat{\bar{u}})(t,\xi_6)\,dt,
\end{align*}
where for $j=4,5,6$, $\tilde{\psi}_{N_j}(\xi_j):=\tilde{\psi}_j(\xi_j/N_j)$ and $\tilde{\psi}_j\in C^\infty_0(\mathbb{R})$ is chosen so that $\tilde{\psi}_j(\cdot /N_j)\equiv 1$ on $\mathrm{supp}\,(\psi_{N_j})$ and $\mathrm{supp}\,(\tilde{\psi}_j(\cdot /N_j))\subset \mathcal{I}_{N_j}$, with $\psi_{N_j}$, $\mathcal{I}_{N_j}$ defined as in Definition~\ref{defn:spaces}.
In the following, we write $N^*_1,\dots ,N^*_6$ to denote the numbers $N_1,\dots ,N_6$ rearranged in decreasing order.
Note that the range of $N_1,\dots ,N_6,N$ can be restricted to
\begin{equation}\label{cond:Nj}
N^*_1\sim N^*_2\quad \text{and}\quad N\lesssim \min \big\{ \max \{ N_1,N_2,N_3\},\,\max \{ N_4,N_5,N_6\} \big\} .
\end{equation}
Applying Lemma~\ref{lem:separation} for each $(N, N_4, N_5, N_6)$, we have
\begin{align*}
&\int _0^T \mathcal{R}(u(t))\,dt\\
&=\sum _{N_1,\dots ,N_6,N\geq 1}\sum _{k,k_4,k_5,k_6\in \mathbb{Z}}\hat{c}_{N,N_4,N_5,N_6}(k,k_4,k_5,k_6)\int_0^T \sum_{\Gamma_6}\mathrm{sgn}(\xi_{12})\mathrm{sgn}(\xi_{56})\\
&\qquad\qquad \times e^{i\big( \frac{k}{N}\xi_1+\frac{k}{N}\xi_2+\frac{k}{N}\xi_3+\frac{k_4}{N_4}\xi_4+\frac{k_5}{N_5}\xi_5+\frac{k_6}{N_6}\xi_6\big)}\hat{u}_1(t,\xi_1)\hat{\bar{u}}_2(t,\xi_2)\cdots \hat{\bar{u}}_6(t,\xi_6)\,dt,
\end{align*}
where we write $u_j=P_{N_j}u$ for brevity, $j=1,\dots ,6$, and 
\begin{equation}\label{multiplier}
\sum _{k,k_4,k_5,k_6\in \mathbb{Z}}\big| \hat{c}_{N,N_4,N_5,N_6}(k,k_4,k_5,k_6)\big| \lesssim \frac{Na(\max \{ N,N_4,N_5,N_6\})}{\max \{ N,N_4,N_5,N_6\}}.
\end{equation}
Since multiplication by $e^{i\theta \xi}$ on the Fourier side does not change the $F_{N_j}(T)$ norm of $u_j$, the proof is reduced to estimating
\begin{equation*}
\begin{aligned}[b]
&\Big| \int_0^T \sum_{\Gamma_6}\mathrm{sgn}(\xi_{12})\mathrm{sgn}(\xi_{56})\hat{u}_1(t,\xi_1)\hat{\bar{u}}_2(t,\xi_2)\cdots \hat{\bar{u}}_6(t,\xi_6)\,dt\Big| \\
&\quad =\Big| \int_0^T \int_{\mathbb{T}}H(u_1\bar{u}_2)u_3\bar{u}_4H(u_5\bar{u}_6)\,dx\,dt\Big| .
\end{aligned}
\end{equation*}

In order to obtain a bound with the short-time norms, we have to divide the time interval into small sub-intervals of length $\leq (N^*_1)^{-1}$ (denoted by $I$), which gives a factor of $O(N^*_1)$.
The strategy in the previous results on DNLS \cite{G11,S21} is to use two bilinear Strichartz and two $L^\infty$ embeddings if $N^*_1\sim N^*_2\gg N^*_3$ or $N^*_1\sim N^*_3\gg N^*_4$; one bilinear Strichartz, one $L^\infty$ embedding and three $L^6$ Strichartz if $N^*_1\sim N^*_4\gg N^*_5$ or $N^*_1\sim N^*_5\gg N^*_6$; and six $L^6$ Strichartz if $N^*_1\sim N^*_6$.
For KDNLS, there are some cases where the same argument does not work due to the presence of the Hilbert transformations.
For instance, we cannot use the standard biliear Strichartz estimate (Lemma~\ref{lem:bs0}) with only one of $u_1$ and $u_2$ involved.
We can indeed use the modified bilinear Strichartz estimates (Lemmas~\ref{lem:bs1}, \ref{lem:bs2}) instead, but the argument will be even more complicated.

The goal is to prove 
\begin{equation}\label{claim:block}
\begin{aligned}
&\text{R.H.S.~of \eqref{multiplier}}\times N^*_1\Big| \int_I \int_{\mathbb{T}}H(u_1\bar{u}_2)u_3\bar{u}_4H(u_5\bar{u}_6)\,dx\,dt\Big| \\
&\quad \lesssim a(N^*_1)\big( N^*_3N^*_4N^*_5N^*_6\big) ^{\frac14+}\prod _{j=1}^6\| \chi_Iu_j\|_{U^2_\Delta(\mathbb{R})}
\end{aligned}
\end{equation}
for each $N_1,\dots ,N_6,N\geq 1$ satisfying \eqref{cond:Nj} and each interval $I=[a,b)\subset [0,T]$ with $|I|\leq (N^*_1)^{-1}$.
In fact, this is enough to carry out the summations in $N_1,\dots ,N_6$ when $s>1/4$.
For summability in $N$, notice that either $N\sim N^*_1$ or $N\lesssim N^*_3$ holds.

\choice%
{In the rest of the proof, we shall establish \eqref{claim:block}, dividing into several cases.
We will estimate the right-hand side of \eqref{multiplier} roughly by $a(N^*_1)$, except for Case (Vd).}%
{In the rest of the proof, we shall establish \eqref{claim:block}, dividing into the following five cases:
\[ \begin{alignedat}{3}
\text{(I)}&~N^*_1\sim N^*_6; &\qquad \text{(II)}&~N^*_1\sim N^*_5\gg N^*_6; &\qquad \text{(III)}&~N^*_1\sim N^*_4\gg N^*_5; \\
\text{(IV)}&~N^*_1\sim N^*_3\gg N^*_4; &\qquad \text{(V)}&~N^*_1\sim N^*_2\gg N^*_3. && \end{alignedat}\]
We will estimate the right-hand side of \eqref{multiplier} roughly by $a(N^*_1)$, except for a subcase (Vd) of Case (V).
For simplicity, we will see in detail only Case (III) and Case (V).%
\footnote{An extended version of this article can be found in \url{https://arxiv.org/pdf/2209.00000v1.pdf}, where we keep a complete proof for the reader's convenience.}
Case (I) is the easiest, and it is treated by the $L^6$ Strichartz estimate (Lemma~\ref{lem:L6}).
Case (II) is a little more involved, and we need the bilinear Strichartz estimate (Lemma~\ref{lem:bs0}).
In Case (IV) we also use the first modified bilinear Strichartz estimate (Lemma~\ref{lem:bs1}), but the argument is similar to that in Case (III).}%
\choice%
{\medskip
\noindent
\textbf{Case (I)} $N^*_1\sim N^*_6$.

This is the easiest case.
We use H\"older, boundedness of the Hilbert transformation, and the $L^6$ Strichartz estimate (Lemma~\ref{lem:L6}) as follows:
\begin{align*}
\mathcal{I}&:=\Big| \int_I \int_{\mathbb{T}}H(u_1\bar{u}_2)u_3\bar{u}_4H(u_5\bar{u}_6)\,dx\,dt\Big| \\
&\;\leq \| H(u_1\bar{u}_2)\|_{L^3(I;L^3)}\| u_3\| _{L^6(I;L^6)}\| u_4\| _{L^6(I;L^6)}\| H(u_5\bar{u}_6)\|_{L^3(I;L^3)} \\
&\;\lesssim \prod_{j=1}^6\| u_j\| _{L^6(I;L^6)}\lesssim \prod_{j=1}^6\| \chi_Iu_j\|_{U^2_\Delta(\mathbb{R})}.
\end{align*}
Since $N^*_1\sim (N^*_3\cdots N^*_6)^{1/4}$, we obtain \eqref{claim:block}.

\medskip
\noindent
\textbf{Case (II)} $N^*_1\sim N^*_5\gg N^*_6$.

Let us focus on the case $N_6=N^*_6$, but the other cases can be treated in a similar way.
Since $N_5\sim N^*_1\gg N_6$, the frequency for the product $H(u_5\bar{u}_6)$ has the size $\sim N^*_1$.
This allows us to divide the integral as
\begin{align*}
\mathcal{I}&\leq \Big| \int_I \int_{\mathbb{T}}HP_{\gtrsim N^*_1}(u_1\bar{u}_2)u_3\bar{u}_4H(u_5\bar{u}_6)\,dx\,dt\Big| \\
&\quad +\Big| \int_I \int_{\mathbb{T}}HP_{\ll N^*_1}(u_1\bar{u}_2)P_{\gtrsim N^*_1}(u_3\bar{u}_4)H(u_5\bar{u}_6)\,dx\,dt\Big| .
\end{align*}
These two terms can be estimated by H\"older, Lemma~\ref{lem:L6}, bilinear Strichartz (Lemma~\ref{lem:bs0}) and $L^\infty$ embedding.
For instance, the last term is bounded by
\begin{align*}
&\| HP_{\ll N^*_1}(u_1\bar{u}_2)\| _{L^3(I;L^3)}\| P_{\gtrsim N^*_1}(u_3\bar{u}_4)\| _{L^2(I;L^2)}\| H(u_5\bar{u}_6)\|_{L^6(I;L^6)}\\
&\quad \lesssim \prod_{j=1,2,5}\| u_j\|_{L^6(I;L^6)}\cdot \| P_{\gtrsim N^*_1}(u_3\bar{u}_4)\| _{L^2(I;L^2)}\| u_6\|_{L^\infty (I;L^\infty )}\\
&\quad \lesssim \frac{1}{(N^*_1)^{1/2}}(N_6)^{1/2}\prod_{j=1}^6\| \chi_Iu_j\|_{U^2_\Delta(\mathbb{R})}.
\end{align*}
Noting that $(N^*_1)^{1/2}(N_6)^{1/2}\lesssim (N^*_3\cdots N^*_6)^{1/4}$, we obtain \eqref{claim:block}.}%
{%
}%

\medskip 
\noindent
\textbf{Case (III)} $N^*_1\sim N^*_4\gg N^*_5$.
\choice%
{This case is more delicate. }%
{%
}%
We consider the following three subcases, according to which two frequencies are smaller than $N^*_1$.

(IIIa) Each of $\{ N_1,N_2\}$, $\{ N_3,N_4\}$, $\{ N_5,N_6\}$ contains at most one frequency $\ll N^*_1$.
Consider the case $N^*_5=N_1$ and $N^*_6=N_3$ for instance, but the other cases can be treated similarly.
Then, we apply Lemma~\ref{lem:bs0} to $H(u_1\bar{u_2})$, Lemma~\ref{lem:L6} to $u_4,u_5,u_6$ and $L^\infty$ embedding to $u_3$, which yields
\choice{%
\[ \mathcal{I}\lesssim \frac{1}{N_2^{1/2}}N_3^{1/2}\prod_{j=1}^6\| \chi_Iu_j\|_{U^2_\Delta(\mathbb{R})}\lesssim \frac{(N^*_3\cdots N^*_6)^{1/4}}{N^*_1}\prod_{j=1}^6\| \chi_Iu_j\|_{U^2_\Delta(\mathbb{R})}.\]
}{%
\begin{align*}
\mathcal{I}&:=\Big| \int_I \int_{\mathbb{T}}H(u_1\bar{u}_2)u_3\bar{u}_4H(u_5\bar{u}_6)\,dx\,dt\Big| \\
&\;\lesssim \frac{1}{N_2^{1/2}}N_3^{1/2}\prod_{j=1}^6\| \chi_Iu_j\|_{U^2_\Delta(\mathbb{R})}\lesssim \frac{(N^*_3\cdots N^*_6)^{1/4}}{N^*_1}\prod_{j=1}^6\| \chi_Iu_j\|_{U^2_\Delta(\mathbb{R})}.
\end{align*}
}%

(IIIb) $\{ N^*_5,N^*_6\} =\{ N_1,N_2\}$ or $\{ N_5,N_6\}$.
Consider the former case, for instance.
We apply the first modified bilinear Strichartz estimate (Lemma~\ref{lem:bs1}) for $H(u_1\bar{u}_2)u_3$ and Lemma~\ref{lem:L6} for the other three $u_j$'s to obtain
\[ \mathcal{I}\lesssim \frac{(N_1\wedge N_2)^{1/2}}{N_3^{1/2}}\prod_{j=1}^6\| \chi_Iu_j\|_{U^2_\Delta(\mathbb{R})}\lesssim \frac{(N^*_3\cdots N^*_6)^{1/4}}{N^*_1}\prod_{j=1}^6\| \chi_Iu_j\|_{U^2_\Delta(\mathbb{R})}.\]

(IIIc) $\{ N^*_5,N^*_6\} =\{ N_3,N_4\}$.
Without loss of generality, we assume $N^*_5=N_3$ and $N^*_6=N_4$.

(IIIc-i) $N_3\sim N_4$. In this case, we divide $H(u_1\bar{u}_2)$ into two parts as follows:
\[ H(u_1\bar{u}_2)=HP_{\gg N_3}(u_1\bar{u}_2)+HP_{\lesssim N_3}(u_1\bar{u}_2).\]
For the first term, we can eliminate the Hilbert transformations by the fact that the frequency for $H(u_5\bar{u}_6)$ must be much bigger than that of $u_3\bar{u}_4$:
\begin{align*}
&\int_I \int_{\mathbb{T}}HP_{\gg N_3}(u_1\bar{u}_2)u_3\bar{u}_4H(u_5\bar{u}_6)\,dx\,dt \\
&\quad =-\int_I \int_{\mathbb{T}}P_{\gg N_3}(u_1\bar{u}_2)H\big[ u_3\bar{u}_4H(u_5\bar{u}_6)\big] \,dx\,dt \\
&\quad =-\int_I \int_{\mathbb{T}}P_{\gg N_3}(u_1\bar{u}_2)u_3\bar{u}_4H^2(u_5\bar{u}_6)\,dx\,dt\\
&\quad =\int_I \int_{\mathbb{T}}P_{\gg N_3}(u_1\bar{u}_2)u_3\bar{u}_4u_5\bar{u}_6\,dx\,dt.
\end{align*}
Now, the desired bound can be obtained by applying Lemma~\ref{lem:bs0} to $u_3u_5$, $L^\infty$ embedding to $u_4$ and Lemma~\ref{lem:L6} to the others, for instance.
To estimate the contribution from the second term, we use the second modified bilinear Strichartz estimate (Lemma~\ref{lem:bs2}) for $HP_{\lesssim N_3}(u_1\bar{u}_2)u_3$ and Lemma~\ref{lem:L6} for $u_4,u_5,u_6$.
In each case, we obtain the factor $a(N^*_1)(N^*_1N^*_6)^{1/2}\sim a(N^*_1)(N^*_3\cdots N^*_6)^{1/4}$.

(IIIc-ii) $N_3\gg N_4$.
In this case, we make a finer decomposition:
\begin{align*}
&H(u_1\bar{u}_2)u_3\bar{u}_4H(u_5\bar{u}_6)\\
&=HP_{\ll N_3}(u_1\bar{u}_2)u_3\bar{u}_4H(u_5\bar{u}_6)+H(u_1\bar{u}_2)u_3\bar{u}_4HP_{\ll N_3}(u_5\bar{u}_6)\\
&\quad -HP_{\ll N_3}(u_1\bar{u}_2)u_3\bar{u}_4HP_{\ll N_3}(u_5\bar{u}_6)\\
&\quad +HP_{\gg N_3}(u_1\bar{u}_2)u_3\bar{u}_4HP_{\gtrsim N_3}(u_5\bar{u}_6)+HP_{\sim N_3}(u_1\bar{u}_2)u_3\bar{u}_4HP_{\gtrsim N_3}(u_5\bar{u}_6).
\end{align*}
There is no contribution from the third term, while the estimate for the fourth term is exactly the same as the first term in Case (IIIc-i), since in the integral we can replace $P_{\gtrsim N_3}(u_5\bar{u}_6)$ by $u_5\bar{u}_6$.
For the first two terms, we can separate two functions of high frequency from the Hilbert transformation; for instance, 
\begin{align*}
&\int_I \int_{\mathbb{T}}HP_{\ll N_3}(u_1\bar{u}_2)u_3\bar{u}_4H(u_5\bar{u}_6)\,dx\,dt \\
&\quad =-\int_I \int_{\mathbb{T}}H\big[ HP_{\ll N_3}(u_1\bar{u}_2)u_3\bar{u}_4\big] u_5\bar{u}_6\,dx\,dt \\
&\quad =-\int_I \int_{\mathbb{T}}HP_{\ll N_3}(u_1\bar{u}_2)(Hu_3)\bar{u}_4u_5\bar{u}_6\,dx\,dt.
\end{align*}
This can again be treated similarly to the first term in Case (IIIc-i).
To estimate the contribution from the last term, we first notice that the frequency for $u_5\bar{u}_6$ must stay $\sim N_3$ in the integral; hence let us write it as $\tilde{P}_{\sim N_3}(u_5\bar{u}_6)$.
We estimate the integral in two ways.
First, we apply Lemma~\ref{lem:bs0} to $P_{\sim N_3}(u_1\bar{u_2})$ and $\tilde{P}_{\sim N_3}(u_5\bar{u}_6)$ and $L^\infty$ embedding to $u_3,u_4$, which yields
\[ \Big| \int_I \int_{\mathbb{T}}HP_{\sim N_3}(u_1\bar{u}_2)u_3\bar{u}_4H\tilde{P}_{\sim N_3}(u_5\bar{u}_6)\,dx\,dt\Big| \lesssim \frac{1}{N_3^{1/2}}\frac{1}{N_3^{1/2}}N_3^{1/2}N_4^{1/2}\prod_{j=1}^6\| \chi_Iu_j\|_{U^2_\Delta(\mathbb{R})}.\]
Next, applying Lemma~\ref{lem:bs2} to $HP_{\sim N_3}(u_1\bar{u_2})u_3$ and $H\tilde{P}_{\sim N_3}(u_5\bar{u}_6)\bar{u}_4$, we have
\[ \Big| \int_I \int_{\mathbb{T}}HP_{\sim N_3}(u_1\bar{u}_2)u_3\bar{u}_4H\tilde{P}_{\sim N_3}(u_5\bar{u}_6)\,dx\,dt\Big| \lesssim \frac{N_3^{1/2}}{N_1^{1/2}}\frac{N_3^{1/2}}{N_5^{1/2}}\prod_{j=1}^6\| \chi_Iu_j\|_{U^2_\Delta(\mathbb{R})}.\]
Interpolating these estimates, we obtain
\begin{align*}
\Big| \int_I \int_{\mathbb{T}}HP_{\sim N_3}(u_1\bar{u}_2)u_3\bar{u}_4H\tilde{P}_{\sim N_3}(u_5\bar{u}_6)\,dx\,dt\Big| &\lesssim \frac{N_3^{1/4}N_4^{1/4}}{(N^*_1)^{1/2}}\prod_{j=1}^6\| \chi_Iu_j\|_{U^2_\Delta(\mathbb{R})}\\
&\sim \frac{(N^*_3\cdots N^*_6)^{1/4}}{N^*_1}\prod_{j=1}^6\| \chi_Iu_j\|_{U^2_\Delta(\mathbb{R})},
\end{align*}
as desired.

\choice%
{\medskip
\noindent
\textbf{Case (IV)} $N^*_1\sim N^*_3\gg N^*_4$.
We consider the following three cases separately.

(IVa) Each of $\{ N_1,N_2\}$, $\{ N_3,N_4\}$ and $\{ N_5,N_6\}$ has exactly one frequency $\sim N^*_1$.
In this case, we can break the binding by the Hilbert transformation among all the six functions; for instance, if $N_1\sim N^*_1\gg N_2$ then $H(u_1\bar{u}_2)=(Hu_1)\bar{u}_2$.
The same argument as for DNLS in \cite{S21} can be applied; namely, we use Lemma~\ref{lem:bs0} twice for the functions corresponding to $N^*_1,\dots ,N^*_4$ and the $L^\infty$ embedding twice for the functions corresponding to $N^*_5,N^*_6$.
We need to consider a product of two high-frequency functions, but this is possible by applying \eqref{est:bs0+} with $K\sim N^*_1$ for the product of $u\bar{u}$ type and by dividing each high-frequency functions into positive-  and negative-frequency parts and applying the last estimate in Remark~\ref{rem:bs} for the product of $uu$ type (note that it is impossible to have the same sign for all of three high frequencies).

(IVb) $N_3\sim N_4\sim N^*_1$.
Without loss of generality, let us assume $N_1\sim N^*_1\gg N_2,N_5,N_6$, so that $H(u_1\bar{u}_2)=(Hu_1)\bar{u}_2$.
We also observe that $u_3\bar{u}_4=P_{\gtrsim N^*_1}(u_3\bar{u}_4)$ in the integral.
Then, we use Lemma~\ref{lem:bs0} for $u_3\bar{u}_4$, Lemma~\ref{lem:bs1} for $(Hu_1)H(u_5\bar{u}_6)$ and $L^\infty$ embedding for $u_2$ to obtain
\[ \mathcal{I}\lesssim \frac{1}{(N^*_1)^{1/2}}\frac{(N_5\wedge N_6)^{1/2}}{N_1^{1/2}}N_2^{1/2}\prod_{j=1}^6\| \chi_Iu_j\|_{U^2_\Delta(\mathbb{R})}\lesssim \frac{(N^*_3\cdots N^*_6)^{1/4}}{N^*_1}\prod_{j=1}^6\| \chi_Iu_j\|_{U^2_\Delta(\mathbb{R})}.\]

(IVc) $N_1\sim N_2\sim N^*_1$ or $N_5\sim N_6\sim N^*_1$.
Consider the former case.
Again, we see $u_1\bar{u}_2=P_{\gtrsim N^*_1}(u_1\bar{u}_2)$ in the integral, so we use Lemma~\ref{lem:bs0} for $u_1\bar{u}_2$.
If $N_5$ or $N_6\sim N^*_1$, then $H(u_5\bar{u}_6)$ can be rephrased as $(Hu_5)\bar{u}_6$ or $u_5H\bar{u}_6$ and we use Lemma~\ref{lem:bs0} (for the product of functions of frequencies $N^*_1,N^*_4$) and two $L^\infty$ embeddings (for the functions of frequencies $N^*_5,N^*_6$).
If $N_3$ or $N_4\sim N^*_1$, we use Lemma~\ref{lem:bs1} and one $L^\infty$ embedding, just as in the preceding case.
In each case, we obtain the same bound as in Case (IVb).}%
{%
}%

\medskip
\noindent
\textbf{Case (V)} $N^*_1\sim N^*_2\gg N^*_3$.
Let us divide into the following four cases.

(Va) One of $\{ N_1,N_2\}$ and one of $\{ N_5,N_6\}$ are comparable to $N^*_1$.
In this case, we can break the binding by the Hilbert transformation between $u_1$ and $\bar{u}_2$ and between $u_5$ and $\bar{u}_6$.
Applying Lemma~\ref{lem:bs0} twice to the pairs $(N^*_1,N^*_3)$ and $(N^*_2,N^*_4)$ and the $L^\infty$ embedding to the functions corresponding to $N^*_5,N^*_6$, we obtain
\[ \mathcal{I}\lesssim \frac{(N^*_5N^*_6)^{1/2}}{N^*_1}\prod_{j=1}^6\| \chi_Iu_j\|_{U^2_\Delta(\mathbb{R})}\lesssim \frac{(N^*_3\cdots N^*_6)^{1/4}}{N^*_1}\prod_{j=1}^6\| \chi_Iu_j\|_{U^2_\Delta(\mathbb{R})}.\]

(Vb) One of $\{ N_1,N_2,N_5,N_6\}$ and one of $\{ N_3,N_4\}$ are comparable to $N^*_1$.
Assume, say, $N_1\sim N_3\sim N^*_1$.
We can deal with $u_1$ and $\bar{u}_2$ of $H(u_1\bar{u}_2)$ separately, which allows us to treat two cases $N_2\geq N_4$ and $N_2\leq N_4$ in a parallel manner (let us assume $N_2\leq N_4$, say).
Apply Lemma~\ref{lem:bs0} to $u_3\bar{u}_4$, Lemma~\ref{lem:bs1} to $(Hu_1)H(u_5\bar{u}_6)$ and $L^\infty$ embedding to $\bar{u}_2$, then we have
\[ \mathcal{I}\lesssim \frac{1}{N_3^{1/2}}\frac{(N_5\wedge N_6)^{1/2}}{N_1^{1/2}}N_2^{1/2}\prod_{j=1}^6\| \chi_Iu_j\|_{U^2_\Delta(\mathbb{R})}\lesssim \frac{(N^*_3\cdots N^*_6)^{1/4}}{N^*_1}\prod_{j=1}^6\| \chi_Iu_j\|_{U^2_\Delta(\mathbb{R})}.\]

(Vc) $N_3\sim N_4\sim N^*_1$.
It suffices to apply Lemma~\ref{lem:bs1} twice to $H(u_1\bar{u}_2)u_3$ and $\bar{u}_4H(u_5\bar{u}_6)$, which gives
\[ \mathcal{I}\lesssim \frac{(N_1\wedge N_2)^{1/2}}{N_3^{1/2}}\frac{(N_5\wedge N_6)^{1/2}}{N_4^{1/2}}\prod_{j=1}^6\| \chi_Iu_j\|_{U^2_\Delta(\mathbb{R})}\lesssim \frac{(N^*_3\cdots N^*_6)^{1/4}}{N^*_1}\prod_{j=1}^6\| \chi_Iu_j\|_{U^2_\Delta(\mathbb{R})}.\]

(Vd) $N_1\sim N_2\sim N^*_1$ or $N_5\sim N_6\sim N^*_1$.
This is the hardest case, and let us focus on the former situation $N_1\sim N_2\gg N_3,\dots ,N_6$.
First, using the conditions \eqref{property:a} on $a$ and \eqref{cond:Nj} on $N$, we deduce that
\begin{equation}\label{multiplier2}
\text{R.H.S.~of \eqref{multiplier}}\lesssim a(N^*_1)\Big( \frac{N^*_3}{N^*_1}\Big) ^{1/2}.
\end{equation}
(In fact, it is only in the case $N_1\sim N_2\sim N^*_1$ that we need to exploit the growth condition on $a$ --- the third line in \eqref{property:a}.)
In the following, we only consider the case $N_3\geq N_4$, $N_5\geq N_6$; the other cases are parallel, though.

(Vd-i) $N^*_3\gg N^*_4$, namely, only one of $N_3,\dots ,N_6$ is much bigger than the others. 
In this case, we can move the Hilbert transformation on $u_1\bar{u}_2$ to one of $u_3,\dots ,u_6$; if $N_3=N^*_3$ we have
\begin{align*}
\int_I \int_{\mathbb{T}}H(u_1\bar{u}_2)u_3\bar{u}_4H(u_5\bar{u}_6)\,dx\,dt
&=-\int_I \int_{\mathbb{T}}u_1\bar{u}_2H(u_3)\bar{u}_4H(u_5\bar{u}_6)\,dx\,dt,
\end{align*}
and if $N_5=N^*_3$ we have
\begin{align*}
\int_I \int_{\mathbb{T}}H(u_1\bar{u}_2)u_3\bar{u}_4H(u_5\bar{u}_6)\,dx\,dt&=-\int_I \int_{\mathbb{T}}u_1\bar{u}_2u_3\bar{u}_4H^2(u_5)\bar{u}_6\,dx\,dt\\
&=\int_I \int_{\mathbb{T}}u_1\bar{u}_2u_3\bar{u}_4u_5\bar{u}_6\,dx\,dt.
\end{align*}
Then, we can obtain 
\[ \mathcal{I}\lesssim \frac{(N^*_6N^*_4)^{1/2}}{N^*_1}\prod_{j=1}^6\| \chi_Iu_j\|_{U^2_\Delta(\mathbb{R})}\lesssim \frac{(N^*_3\cdots N^*_6)^{1/4}}{N^*_1}\prod_{j=1}^6\| \chi_Iu_j\|_{U^2_\Delta(\mathbb{R})}\]
by using Lemmas~\ref{lem:bs0} and \ref{lem:bs1} when $N_3=N^*_3$, or Lemma~\ref{lem:bs0} twice when $N_5=N^*_3$.
Note that we do not actually need the improved bound \eqref{multiplier2} in this subcase.

(Vd-ii) $N^*_3\sim N^*_6$.
In this case, we use Lemma~\ref{lem:bs2} for $H(u_1\bar{u}_2)u_3$ (noticing that $u_1\bar{u}_2$ may be replaced by $P_{\lesssim N^*_3}(u_1\bar{u}_2)$), Lemma~\ref{lem:L6} for the others, to obtain 
\[ \mathcal{I}\lesssim \frac{(N^*_3)^{1/2}}{(N^*_1)^{1/2}}\prod_{j=1}^6\| \chi_Iu_j\|_{U^2_\Delta(\mathbb{R})}.\]
Combining it with \eqref{multiplier2}, we have the desired estimate.

(Vd-iii) $N^*_3\sim N^*_4\gg N^*_6$ and $N_5\gg N_6$.
In this case $H(u_5\bar{u}_6)=(Hu_5)\bar{u}_6$, so that we can separate $u_3,\dots ,u_6$.

(Vd-iii-1) If $N^*_3\sim N^*_4\gg N^*_5$, we use Lemma~\ref{lem:bs2} for $H(u_1\bar{u}_2)u^*_3$, Lemma~\ref{lem:bs0} for $u^*_4u^*_5$, and the $L^\infty$ embedding for $u^*_6$, where $u^*_j$ means the function corresponding to $N^*_j$.
The resulting eatimate is
\[ \mathcal{I}\lesssim \frac{(N^*_3)^{1/2}}{(N^*_1)^{1/2}}\frac{1}{(N^*_4)^{1/2}}(N^*_6)^{1/2}\prod_{j=1}^6\| \chi_Iu_j\|_{U^2_\Delta(\mathbb{R})}\lesssim \frac{(N^*_5N^*_6)^{1/4}}{(N^*_1)^{1/2}}\prod_{j=1}^6\| \chi_Iu_j\|_{U^2_\Delta(\mathbb{R})},\]
which is sufficient together with \eqref{multiplier2}. 

(Vd-iii-2) If $N^*_3\sim N^*_5\gg N^*_6$, by the assumption we have $N_3\sim N_4\sim N_5\gg N_6$.
In this case, we make the decomposition 
\[ u_3\bar{u}_4=P_{\ll N_3}(u_3\bar{u}_4)+P_{\sim N_3}(u_3\bar{u}_4).\]
For the first term, we can separate $u_1$ and $u_2$ as
\begin{align*}
\int_I \int_{\mathbb{T}}H(u_1\bar{u}_2)P_{\ll N_3}(u_3\bar{u}_4)H(u_5)\bar{u}_6\,dx\,dt&=-\int_I \int_{\mathbb{T}}u_1\bar{u}_2H\big[ P_{\ll N_3}(u_3\bar{u}_4)H(u_5)\bar{u}_6\big] \,dx\,dt\\
&=\int_I \int_{\mathbb{T}}u_1\bar{u}_2P_{\ll N_3}(u_3\bar{u}_4)u_5\bar{u}_6\,dx\,dt.
\end{align*}
Hence, this is similar to Case (Vd-i) and easily treated by applying Lemmas~\ref{lem:bs0} and \ref{lem:bs1}.
For the second term, we use Lemma~\ref{lem:bs2} for $H(u_1\bar{u}_2)H(u_5)$, Lemma~\ref{lem:bs0} for $P_{\sim N_3}(u_3\bar{u}_4)$ and $L^\infty$ embedding for $\bar{u}_6$, to obtain 
\[ \mathcal{I}\lesssim \frac{N_3^{1/2}}{N_1^{1/2}}\frac{1}{N_3^{1/2}}N_6^{1/2}\prod_{j=1}^6\| \chi_Iu_j\|_{U^2_\Delta(\mathbb{R})}\lesssim \frac{(N^*_5N^*_6)^{1/4}}{(N^*_1)^{1/2}}\prod_{j=1}^6\| \chi_Iu_j\|_{U^2_\Delta(\mathbb{R})},\]
which is again sufficient.

(Vd-iv) $N^*_3\sim N^*_4\gg N^*_6$ and $N_5\sim N_6$.

(Vd-iv-1) If $N_3\gg N_5$, so $N_3\sim N_4\gg N_5\sim N_6$, we use Lemma~\ref{lem:bs2} for $H(u_1\bar{u}_2)u_3$ and Lemma~\ref{lem:bs1} for $\bar{u}_4H(u_5\bar{u}_6)$.
We obtain
\[ \mathcal{I}\lesssim \frac{N_3^{1/2}}{N_1^{1/2}}\frac{N_5^{1/2}}{N_4^{1/2}}\prod_{j=1}^6\| \chi_Iu_j\|_{U^2_\Delta(\mathbb{R})}\lesssim \frac{(N^*_5N^*_6)^{1/4}}{(N^*_1)^{1/2}}\prod_{j=1}^6\| \chi_Iu_j\|_{U^2_\Delta(\mathbb{R})}.\]

(Vd-iv-2) If $N_5\gtrsim N_3\sim N_4$, namely, $N_5\sim N_6\gg N_3\sim N_4$, we make a decomposition similar to that in Case (Vdiii-2):
\[ H(u_5\bar{u}_6)=HP_{\gg N_3}(u_5\bar{u}_6)+HP_{\lesssim N_3}(u_5\bar{u}_6).\]
For the first term, we can separate $u_1$ and $u_2$ as in Case (Vd-iii-2), since $H[u_3\bar{u}_4HP_{\gg N_3}(u_5\bar{u}_6)]=u_3\bar{u}_4H^2P_{\gg N_3}(u_5\bar{u}_6)$.
The estimate is then easy and similar to Case (Vd-i).
For the second term, we can put $P_{\lesssim N_3}$ also on $H(u_1\bar{u}_2)$.
Applying Lemma~\ref{lem:bs2} to $HP_{\lesssim N_3}(u_1\bar{u}_2)u_3$ and Lemma~\ref{lem:L6} to the others, we have
\[ \mathcal{I}\lesssim \frac{N_3^{1/2}}{N_1^{1/2}}\prod_{j=1}^6\| \chi_Iu_j\|_{U^2_\Delta(\mathbb{R})}\lesssim \frac{(N^*_5N^*_6)^{1/4}}{(N^*_1)^{1/2}}\prod_{j=1}^6\| \chi_Iu_j\|_{U^2_\Delta(\mathbb{R})}.\]

(Vd-iv-3) The only remaining case is $N_5\gtrsim N_3\gg N_4$, namely, either $N_5\sim N_6\sim N_3\gg N_4$ or $N_5\sim N_6\gg N_3\gg N_4$.
We make a slightly finer decomposition
\[ H(u_5\bar{u}_6)=HP_{\ll N_3}(u_5\bar{u}_6)+\sum _{N_3\lesssim K\lesssim N_5}HP_K(u_5\bar{u}_6).\]
The first term is again easy to treat, since we can separate $u_1$ and $u_2$ by the identity $H[u_3\bar{u}_4HP_{\ll N_3}(u_5\bar{u}_6)]=(Hu_3)\bar{u}_4HP_{\ll N_3}(u_5\bar{u}_6)$.
For the second term, for each $K$ we can put $P_{\lesssim K}$ on $H(u_1\bar{u}_2)$.
Hence, by applying Lemma~\ref{lem:bs2} to $HP_{\lesssim K}(u_1\bar{u}_2)u_3$ and using Lemma~\ref{lem:bs0} for $HP_K(u_5\bar{u}_6)$, $L^\infty$ embedding for $\bar{u}_4$, we obtain
\[ \mathcal{I}\lesssim \sum _{N_3\lesssim K\lesssim N_5}\frac{K^{1/2}}{N_1^{1/2}}\frac{1}{K^{1/2}}N_4^{1/2}\prod_{j=1}^6\| \chi_Iu_j\|_{U^2_\Delta(\mathbb{R})}\lesssim (N^*_3)^{0+}\frac{(N^*_5N^*_6)^{1/4}}{(N^*_1)^{1/2}}\prod_{j=1}^6\| \chi_Iu_j\|_{U^2_\Delta(\mathbb{R})},\]
which together with \eqref{multiplier2} shows the desired estimate \eqref{claim:block}.

We have thus completed the case-by-case analysis for the proof of \eqref{claim:block}.
\end{proof}


\noindent
{\bf Acknowledgements}.
The first author N.K is partially supported by JSPS KAKENHI Grant-in-Aid for Young Researchers (B) (16K17626) and Grant-in-Aid for Scientific Research (C) (20K03678).
The second author Y.T is partially supported by JSPS KAKENHI Grant-in-Aid for Scientific Research (B) (17H02853).


\end{document}